\date{\today}
\documentclass[10pt]{amsart}
\usepackage{amssymb,amsfonts,amsthm,amsmath,verbatim,lscape,color,tensor,url}
\usepackage[margin=1.5in]{geometry}
\usepackage[enableskew]{youngtab}
\usepackage[all]{xy}
\usepackage[nameinlink,capitalise,noabbrev]{cleveref}

\newtheorem{thm}[subsection]{Theorem}
\newtheorem{prop}[subsection]{Proposition}
\newtheorem{cor}[subsection]{Corollary}
\newtheorem{lemma}[subsection]{Lemma}

\newtheorem*{mainthm}{Theorem}

\theoremstyle{definition}
\newtheorem{definition}[subsection]{Definition}

\newtheorem{example}[subsection]{Example}

\numberwithin{equation}{section}

\newcommand{\powser}[1]{[\![#1]\!]}

\newcommand{\G}{\mathbb{G}}
\newcommand{\F}{\mathbb{F}}
\newcommand{\A}{\hat{\mathbb{A}}^{1}}

\newcommand{\Sect}{\mathcal{O}}
\newcommand{\Q}{\mathbb{Q}}
\newcommand{\Po}{\mathbb{P}}
\newcommand{\C}{\mathbb{C}}
\newcommand{\Z}{\mathbb{Z}}
\newcommand{\N}{\mathbb{N}}
\newcommand{\QZ}[1]{(\Q_p/\Z_p)^{#1}}
\newcommand{\Zp}[1]{\Z/p^{#1}}

\newcommand{\m}{\mathfrak{m}}
\newcommand{\n}{\mathfrak{n}}
\newcommand{\al}{\alpha}

\newcommand{\lra}[1]{\overset{#1}{\longrightarrow}}

\newcommand{\Prod}[1]{\underset{#1}{\prod}}
\newcommand{\Oplus}[1]{\underset{#1}{\bigoplus}}

\newcommand{\um}{\underline{m}}

\newcommand{\ka}{\kappa}
\newcommand{\Gu}{\G_{u}}
\newcommand{\Ralg}{R\text{-alg}}

\DeclareMathOperator{\Aut}{Aut}

\DeclareMathOperator{\im}{im}

\DeclareMathOperator{\Hom}{Hom}
\DeclareMathOperator{\colim}{colim}

\DeclareMathOperator{\Isog}{Isog}
\DeclareMathOperator{\Spec}{Spec}
\DeclareMathOperator{\Spf}{Spf}

\DeclareMathOperator{\Sub}{Sub}
\DeclareMathOperator{\Div}{Div}
\DeclareMathOperator{\Level}{Level}

\DeclareMathOperator{\Sum}{Sum}
\DeclareMathOperator{\Tr}{Tr}
\DeclareMathOperator{\tr}{tr}
\DeclareMathOperator{\GL}{GL}
\DeclareMathOperator{\Comprings}{CompRings}
\DeclareMathOperator{\LT}{LT}
\DeclareMathOperator{\Cl}{Cl}
\DeclareMathOperator{\Groupoids}{Groupoids}
\DeclareMathOperator{\Frob}{Frob}
\DeclareMathOperator{\FGL}{FGL}
\DeclareMathOperator{\FG}{FG}

\DeclareMathOperator{\lat}{\mathbb{L}}
\DeclareMathOperator{\Gal}{Gal}
\DeclareMathOperator{\trans}{trans}
\DeclareMathOperator{\latpk}{\lat \mkern-3mu / \mkern-3mu \lat_{p^k}}
\DeclareMathOperator{\qz}{\mathbb{T}}
\DeclareMathOperator{\qzpk}{\qz[p^k]}

\DeclareMathOperator{\id}{id}
\newcommand{\upperRomannumeral}[1]{\uppercase\expandafter{\romannumeral#1}}

\begin{document}
\title{Lubin-Tate Theory, Character Theory, and Power Operations}

\author{Nathaniel Stapleton}
\address{University of Kentucky, Lexington, KY}
\email{nat.j.stapleton@gmail.com}

\begin{abstract}    
This expository paper introduces several ideas in chromatic homotopy theory around Morava's extraordinary $E$-theories. In particular, we construct various moduli problems closely related to Lubin-Tate deformation theory and study their symmetries. These symmetries are then used in conjunction with the Hopkins-Kuhn-Ravenel character theory to provide formulas for the power operations and the stabilizer group action on the $E$-cohomology of a finite group. 
\end{abstract}

\maketitle


\section{Introduction} 
Chromatic homotopy theory decomposes the category of spectra at a prime $p$ into a collection of categories according to certain periodicities. There is one of these categories for each natural number $n$ and it is called the $K(n)$-local category. For $n>1$, the objects in the $K(n)$-local category are often quite difficult to understand computationally. When $n=2$, significant progress has been made but above $n=2$ many basic computational questions are open. Even the most well-behaved ring spectra in the $K(n)$-local category, the height $n$ Morava $E$-theories, still hold plenty of mysteries. In order to understand $E$-cohomology, Hopkins, Kuhn, and Ravenel constructed a character map for each $E$-theory landing in a form of rational cohomology. They proved that the codomain of their character map serves as an approximation to $E$-cohomology in a precise sense. It turns out that many of the deep formal properties of the $K(n)$-local category, when specialized to the Morava $E$-theories, can be expressed in terms of simple formulas or relations satisfied by simple formulas on the codomain of the character map. This is intriguing for several reasons: The codomain of the character map is not $K(n)$-local, it is a rational approximation to $E$-cohomology so it removes all of the torsion from $E$-cohomology. The codomain of the character map for the $E$-cohomology of a finite group is a simple generalization of the ring of class functions on the group. Thus, in this case, these deep properties of the $K(n)$-local category are reflected in combinatorial and group theoretic properties of certain types of conjugacy classes in the group. Finally, a certain $\Q$-algebra, known as the Drinfeld ring of infinite level structures, arises from topological considerations as the coefficients of the codomain of the character map. This $\Q$-algebra plays an important role in local Langlands and the group actions that feed into local Langlands are closely related to fundamental properties of the Morava $E$-theories.

The primary purpose of this article is to explain the relationship between the power operations 
\[
P_m \colon E^0(BG) \to E^0(BG \times B\Sigma_m)
\]
on the Morava $E$-cohomology of finite groups and Hopkins-Kuhn-Ravenel character theory. Morava $E$-theory is built out of the Lubin-Tate ring associated to a finite height formal group over a perfect field of characteristic $p$ by using the Landweber exact functor theorem. Lubin-Tate theory plays an important role in local arithmetic geometry and so it is not too surprising that other important objects from arithmetic geometry, such as the Drinfeld ring, that are closely related to the Lubin-Tate ring arise in the construction of the character map. The power operations on Morava $E$-theory are a consequence of the Goerss-Hopkins-Miller theorem that implies that each Morava $E$-theory spectrum has an essentially unique $E_{\infty}$-ring structure. However, this multiplicative structure plays no role in the construction of the character map, so it should be surprising that the two structures interact as well as they do. The first indicator that this might be the case is the fact that the Drinfeld ring at infinite level picks up extra symmetry that does not exist at finite level. This extra symmetry turns out to be the key to understanding the relationship between character theory and power operations.

To give a careful statement of the results, we need some setup. Fix a prime $p$ and a height $n$ formal group $\F$ over $\ka$, a perfect field of characteristic $p$. Associated to this data is an $E_{\infty}$-ring spectrum $E = E(\F, \ka)$ known as Morava $E$-theory. The coefficient ring $E^0$ is the Lubin-Tate ring. It has the astounding property that the automorphism group of $E$ as an $E_{\infty}$-ring spectrum is the discrete extended Morava stabilizer group $\Aut(\F/\ka) \rtimes \Gal(\ka/\F_p)$. The homotopy fixed points of $E$ for this action by the extended stabilizer group are precisely the $K(n)$-local sphere - the unit in the category of $K(n)$-local spectra. Therefore, a calculation of the $E$-cohomology of a space is the starting point for the calculation the $K(n)$-local cohomotopy of the space. 

Of particular interest is the $E$-cohomology of finite groups. The reason for this is several fold. Morava $E$-theory is a generalization of $p$-adic $K$-theory to higher height. Thus the $E$-cohomology of a finite group is a natural generalization of the representation ring (at the prime $p$) of the group. Further, there is an intriguing relationship between the $E$-cohomology of finite groups and formal algebraic geometry. This is exemplified in Strickland's theorem, a theorem which produces a canonical isomorphism between a quotient of the $E$-cohomology of the symmetric group and the ring of functions on the scheme that represents subgroups of the universal deformation of $\F$. Finally, it seems reasonable to hope that a good understanding of the these rings could lead to insight into the relationship between chromatic homotopy theory and geometry.

Hopkins, Kuhn, and Ravenel made significant progress towards understanding these rings by constructing an analogue of the character map in representation theory for these rings. Recall that the character map is a map of commutative rings from the representation ring of $G$ into class functions on $G$ taking values in the complex numbers. Hopkins, Kuhn, and Ravenel constructed a generalization of the ring of class functions on $G$, called $\Cl_n(G,C_0)$. The ring $C_0$ is an important object in arithmetic geometry that appears here because of it carries the universal isomorphism of $p$-divisible groups between $\QZ{n}$ and the universal deformation of $\F$. Using the ring of generalized class functions, they constructed a character map
\[
\chi \colon E^0(BG) \to \Cl_n(G,C_0).
\]
The ring $\Cl_n(G,C_0)$ comes equipped with a natural action of $\GL_n(\Z_p)$. Hopkins, Kuhn, and Ravenel prove that the character induces an isomorphism
\[
\Q \otimes E^0(BG) \cong  \Cl_n(G,C_0)^{\GL_n(\Z_p)}.
\]
Thus $\Q \otimes E^0(BG)$ admits a purely algebraic description.

The $E_{\infty}$-ring structure on $E$ gives rise to power operations. For each $m \geq 0$, these are multiplicative non-additive natural transformations
\[
P_m \colon E^0(X) \to E^0(X \times B\Sigma_m). 
\]
When $X = BG$, one might hope for an operation on generalized class functions $\Cl_n(G,C_0) \to \Cl_n(G \times \Sigma_m,C_0)$ compatible with $P_m$ through the character map.


The purpose of this article is to give a formula for such a map. To describe the formula we need to introduce some more notation. Let $\lat = \Z_{p}^n$ and let $\qz = \lat^*$, the Pontryagin dual of $\lat$, so that $\qz$ is abstractly isomorphic to $\QZ{n}$. Let $\Aut(\qz)$ be the automorphisms of $\qz$ and let $\Isog(\qz)$ be the monoid of endoisogenies (endomorphisms with finite kernel) of $\qz$. The group $\Aut(\qz)$ is anti-isomorphic to $\GL_n(\Z_p)$. There is a right action of $\Isog(\qz)$ by ring maps on $C_0$. Given a finite group $G$, let $\hom(\lat, G)$ be the set of continuous homomorphisms from $\lat$ to $G$. This set is in bijective correspondence with $n$-tuples of pairwise commuting $p$-power order elements in $G$. The group $G$ acts $\hom(\lat,G)$ by conjugation and we will write $\hom(\lat,G)_{/\sim}$ for the set of conjugacy classes. The ring of generalized class functions $\Cl_n(G,C_0)$ is just the ring of $C_0$ valued functions on $\hom(\lat,G)_{/\sim}$. There is a canonical bijection
\[
\hom(\lat,\Sigma_m)_{/\sim} \cong \Sum_m(\qz),
\]
where 
\[
\Sum_m(\qz) = \{\oplus_i H_i| H_i \subset \qz \text{ and } \sum_i|H_i|=m\},
\]
formal sums of subgroups of $\qz$ such that the sum of the orders is precisely $m$. Finally let $\Sub(\qz)$ be the set of finite subgroups of $\qz$ and let $\pi \colon \Isog(\qz) \to \Sub(\qz)$ be the map sending an isogeny to its kernel. Choose a section $\phi$ of $\pi$ so that, for $H \subset \qz$, $\phi_H$ is an isogeny with kernel $H$.

Now we are prepared to describe the formula for the power operation on class functions
\[
P_{m}^{\phi} \colon \Cl_n(G,C_0) \to \Cl_n(G \times \Sigma_m,C_0),
\]
which depends on our choice of section $\phi$. Because of the canonical bijection described above, an element in $\hom(\lat, G \times \Sigma_m)_{/\sim}$ may be described as a pair $([\al],\oplus_i H_i)$, where $\al \colon \lat \to G$. Given $f \in \Cl_n(G,C_0)$, we define
\[
P_{m}^{\phi}(f)([\al], \oplus_i H_i) = \Prod{i} f([\al \phi_{H_i}^*]) \phi_{H_i} \in C_0,
\]
where $\phi_{H_i}^* \colon \lat \to \lat$ is the Pontryagin dual of $\phi_{H_i}$. Using this construction, we prove the following theorem:
\begin{mainthm}
For all section $\phi \in \Gamma(\Sub(\qz), \Isog(\qz))$, there is a commutative diagram
\[
\xymatrix{E^0(BG) \ar[r]^-{P_m} \ar[d]_-{\chi} & E^0(BG \times B\Sigma_m) \ar[d]^-{\chi} \\ \Cl_n(G,C_0) \ar[r]^-{P_{m}^{\phi}} & \Cl_n(G \times \Sigma_m,C_0),}
\]
where the vertical arrows are the Hopkins-Kuhn-Ravenel character map.
\end{mainthm}

Several ingredients go into the proof of this theorem. The ring $C_0$ is the rationalization of the Drinfeld ring at infinite level, it is closely related to a certain moduli problem over Lubin-Tate space. The first ingredient is an understanding of the symmetries of this moduli problem. These symmetries play an important role in the local Jacquet-Langlands correspondence, do not go into the proof of the Goerss-Hopkins-Miller theorem, and show up here because of their close relationship to a result of Ando, Hopkins, and Strickland. The second ingredient is the result of Ando, Hopkins, and Strickland, which gives an algebro-geometric description of a simplification of $P_m$ applied to abelian groups. The third ingredient is the fact that, rationally, the $E$-cohomology of finite groups can be detected by the $E$-cohomology of the abelian subgroups. This is a consequence of Hopkins, Kuhn, Ravenel character theory. 

Recall that $\GL_n(\Z_p)$ acts on $\Cl_n(G,C_0)$. We prove that the choice of section $\phi$ disappears after restricting to the $\GL_n(\Z_p)$-invariant class functions. 
\begin{mainthm}
For all $\phi \in \Gamma(\Sub(\qz), \Isog(\qz))$, the power operation $P_{m}^{\phi}$ sends $\GL_n(\Z_p)$-invariant class functions to $\GL_n(\Z_p)$-invariant class functions and the resulting operation
\[
\Cl_n(G,C_0)^{\GL_n(\Z_p)} \to \Cl_n(G \times \Sigma_m,C_0)^{\GL_n(\Z_p)}
\]
is independent of the choice of $\phi$.
\end{mainthm}

Both of these results are simplifications of the main results of \cite{cottpo}. By proving simpler statements, we are able to bypass work understanding conjugacy classes of $n$-tuples of commuting elements in wreath products and also the application of the generalization of Strickland's theorem in \cite{genstrickland}. This simplifies the argument and allows us to make use of more classical results in stable homotopy theory thereby easing the background required of the reader.

A second goal of this article is to describe an action of the stabilizer group on generalized class functions that is compatible with the action of the stabilizer group on $E^0(BG)$ through the character map. There is a natural action of $\Aut(\F/\ka)$ on $C_0$ and we show that the diagonal action of $\Aut(\F,\ka)$ on $\Cl_n(G,C_0)$ makes the character map into an $\Aut(\F/\ka)$-equivariant map. Further, we show that the diagonal action of the stabilizer group on $\Cl_n(G,C_0)$ commutes with the power operations $P_{m}^{\phi}$ exhibiting an algebraic analogue of the fact that $\Aut(\F/\ka)$ acts on $E$ by $E_{\infty}$-ring maps.

%
%

\subsection*{Acknowledgements}
Most of this article is based on \cite{cottpo}, joint work with Tobias Barthel. My understanding of this material is a result of our work together and it is a pleasure to thank him for being a great coauthor. I also offer my sincere thanks to Peter Bonventre and Tomer Schlank for their helpful comments.

\section{Notation and conventions}\label{sec:notation}
We will use $\hom$ to denote the set of morphisms between two objects in a category. We will capitalize the first letter, such as in $\Hom$, $\Level$, and $\Sub$, to denote functors. Given a (possibly formal) scheme $X$, we will write $\Sect_X$ to denote the ring of functions $\hom(X,\mathbb{A}^1)$ on $X$.

\section{Lubin-Tate theory}\label{LT}
In this section we recall the Lubin-Tate moduli problem and describe several well-behaved moduli problems that live over the Lubin-Tate moduli problem. 

We begin with a brief overview of the theory of ($1$-dimensional, commutative) formal groups over complete local rings. Let $\Comprings$ be the category of complete local rings. Objects in this category are local commutative rings $(R, \m)$ such that $R \cong \lim_{k} R/\m^k$ and maps are local homomorphisms of rings. We will denote the quotient map $R \to R/\m$ by $\pi$. Note that if $\ka$ is a field, then $(\ka, (0))$ is a complete local ring. Given a complete local ring $(R, \m)$, let $\Comprings_{R/}$ be the category of complete local $R$-algebras. The formal affine line over $R$, 
\[
\A_R \colon \Comprings_{R/} \to \text{Set}_*,
\]
is the functor to pointed sets corepresented by the complete local $R$-algebra $(R\powser{x}, \m+(x))$. Thus, for a complete local $R$-algebra $(S, \n)$,
\[
\A_R(S,\n) = \hom_{\Ralg}(R\powser{x},S) \cong \n,
\]
where $\hom_{\Ralg}(-,-)$ denotes the set of morphisms in $\Comprings_{R/}$ and the last isomorphism is the map that takes a morphism to the image of $x$ in $S$.

A formal group $\G$ over $R$ is a functor
\[
\G \colon \Comprings_{R/} \to \text{Abelian Groups}
\]
that is abstractly isomorphic to $\A_R$ when viewed as a functor to pointed sets. A coordinate on a formal group $\G$ over $R$ is a choice of isomorphism $\G \lra{\cong} \A_R$. Given a coordinate, the multiplication map
\[
\mu \colon \G \times \G \to \G
\]
gives rise to a map
\[
\mu^* \colon R\powser{x} \to R\powser{x_1, x_2}
\]
and $\mu^*(x) = x_1+_{\G}x_2$ is the formal group law associated to the coordinate. For $i \in \N$, we define the $i$-series to be the power series
\[
[i]_{\G}(x) = \overbrace{x +_{\G} \ldots +_{\G} x}^{i \text{ terms}}.
\]
Given a formal group over a complete local ring $R$ and a map of complete local rings $j \colon R \to S$, we will write $j^*\G$ for the restriction of the functor $\G$ to the category of complete local $S$-algebras.

Let $\ka$ be a perfect field of characteristic $p$. ``Perfect" means that the Frobenius $\sigma \colon \ka \to \ka$ is an isomorphism of fields. A formal group $\F$ over $\ka$ is said to be of height $n$ if, after choosing a coordinate, the associated formal group law $x +_{\F} y$ has the property that the $p$-series $[p]_{\F}(x)$ has first term $x^{p^n}$. 

Fix a formal group $\F$ over $\ka$ of height $n$. All of the constructions in this section will depend on this choice. Associated to $\F / \ka$ is a moduli problem 
\[
\LT \colon \Comprings \to \Groupoids
\]
studied by Lubin and Tate in \cite{lubintate}. We will take a coordinate-free approach to this moduli problem that is closely related to \cite[Section 6]{subgroups}. The functor $\LT$ takes values in groupoids. Next we will describe the objects and morphisms of these groupoids.


A deformation of $\F / \ka$ to a complete local ring $(R,\m)$ is a triple of data
\[
(\G / R, i \colon \ka \to R/\m, \tau: \pi^* \G \lra{\cong} i^* \F),
\]
where $\G$ is a formal group over $R$, $i \colon \ka \to R/\m$ is a map of fields, and $\tau$ is an isomorphism of formal groups over $R/\m$. We will often use shorthand and refer to a deformation as a triple $(\G, i, \tau)$. The deformations of $\F / \ka$ form the objects of $\LT(R,\m)$. The morphisms of $\LT(R,\m)$ are known as $\star$-isomorphisms. A $\star$-isomorphism $\delta$ between two deformations $(\G, i, \tau)$ and $(\G', i', \tau')$ only exists when $i = i'$ and then it is an isomorphism of formal groups $\delta \colon \G \lra{\cong} \G'$ compatible with the isomorphisms $\tau$ and $\tau'$ in the sense that the square
\[
\xymatrix{\pi^* \G \ar[r]^{\tau} \ar[d]_{\pi^* \delta} & i^* \F  \ar[d]^{\id} \\ \pi^* \G' \ar[r]^{\tau'} & i^* \F} 
\]
commutes.

Given a map of complete local rings $j \colon (R,\m) \to (S,\n)$, there is an induced map of groupoids $\LT(R,\m) \to \LT(S,\n)$ defined by sending
\[
(\G, i, \tau) \mapsto (j^* \G, \ka \lra{i} R/\m \lra{j/\m} S/\n, (j/\m)^* \tau). 
\]
This makes $\LT$ into a functor. Let $W(\ka)$ be the ring of $p$-typical Witt vectors on $\ka$. This is a complete local ring of characteristic $0$ with the property that $W(\ka)/p \cong \ka$. When $\ka = \F_p$, $W(\ka) = \Z_p$ (the $p$-adic integers). Now we are prepared to state the Lubin-Tate theorem. This theorem was proved in \cite{lubintate}, the homotopy theorist may be interested in reading \cite[Chapters 4 and 5]{Rezk-Notes} or \cite[Lecture 21]{luriechromatic}.

\begin{thm}\cite{lubintate}  \label{lubintate}
The functor $\LT$ is corepresented by a complete local ring $\Sect_{\LT}$ non-canonically isomorphic to $W(\ka)\powser{u_1, \ldots, u_{n-1}}$ called the Lubin-Tate ring.
\end{thm}

Part of the content of this theorem is that, for any complete local ring $(R,\m)$, the groupoid $\LT(R,\m)$ is fine (ie. there is a unique isomorphism between any two objects that are isomorphic). It follows that $\Sect_{\LT}$ carries the universal deformation $(\Gu, \id_{\ka}, \id_{\F})$. This means that, given any deformation $(\G, i, \tau)$ over a complete local ring $R$ there exists a map $j \colon \Sect_{\LT} \to R$ such that $(j^*\Gu, j/\m, \id_{(j/\m)^* \F})$ is $\star$-isomorphic to $(\G, i, \tau)$.

It is possible to add various bells and whistles to the Lubin-Tate moduli problem. We will now explain three ways of doing this.

We construct a modification of the Lubin-Tate moduli problem depending on a finite abelian group $A$. If $\G$ is a formal group over a complete local ring $(R,\m)$, then a homomorphism $A \to \G$ over $R$ is just a map of abelian groups $A \to \G(R)$. We will write $\hom(A, \G(R))$ for this abelian group. Define 
\[
\Hom(A,\Gu) \colon \Comprings \to \Groupoids
\]
to be the functor with $\Hom(A,\Gu)(R,\m)$ equal to the groupoid with objects triples $(f \colon A \to \G, i, \tau)$, where $f$ is a homomorphism from $A$ to $\G$ over $R$ and $\G$, $i$ and $\tau$ are as before, and morphisms $\star$-isomorphisms that commute with this structure. In other words, a $\star$-isomorphism from $(f \colon A \to \G, i, \tau)$ to $(f' \colon A \to \G', i, \tau')$ is a $\star$-isomorphism $\delta$ from $(\G, i, \tau)$ to $(\G', i, \tau')$ such that the diagram
\[
\xymatrix{A \ar[r]^f \ar[rd]^{f'} & \G \ar[d]^{\delta} \\ & \G'}
\]
commutes.

The functor $\Hom(A,\Gu)$ lives over $\LT$. There is a forgetful natural transformation $\Hom(A,\Gu) \to \LT$ given by sending an object $(f \colon A \to \G, i, \tau)$ to $(\G, i, \tau)$. It follows from the definition that the forgetful functor $\Hom(A,\Gu)(R,\m) \to \LT(R,\m)$ is faithful. Thus $\Hom(A,\Gu)(R,\m)$ is a fine groupoid. 


\begin{cor} \label{HomLT}
The functor $\Hom(A,\Gu)$ is corepresentable by a complete local ring $\Sect_{\Hom(A,\Gu)}$ that is finitely generated and free as a module over $\Sect_{\LT}$.
\end{cor}
\begin{proof}
The idea of the proof is that we will first construct an $\Sect_{\LT}$-algebra with the correct algebraic properties and then use Theorem \ref{lubintate} to show that the underlying complete local ring corepresents the moduli problem $\Hom(A,\Gu)$. 

Recall that $\Comprings_{\Sect_{\LT}/}$ is the category of complete local $\Sect_{\LT}$-algebras. Consider the functor
\[
\Gu^{A} \colon \Comprings_{\Sect_{\LT}/} \to \text{Abelian Groups}
\]
that sends a complete local $\Sect_{\LT}$-algebra $j \colon \Sect_{\LT} \to R$ to the abelian group $\hom(A,j^*\Gu(R))$, which is isomorphic to $\hom(A, \Gu(R))$. When $A = C_{p^k}$, this functor is corepresented by $\Sect_{\LT}\powser{x}/[p^k]_{\Gu}(x)$, where $[p^k]_{\Gu}(x)$ is the $p^k$-series for the formal group law associated a choice of coordinate on $\Gu$. This is a complete local ring with maximal ideal $\m+(x)$, where $\m$ is the maximal ideal of $\Sect_{\LT}$, and it follows from the Weierstrass preparation theorem that it is a free module of rank $p^{kn}$. If $A = \prod_iC_{p^{k_i}}$, then $\Gu^A$ is corepresented by a tensor product of these sorts of $\Sect_{\LT}$-algebras and thus is finitely generated and free as an $\Sect_{\LT}$-module. Thus $\Sect_{\Gu^A}$ is a complete local $\Sect_{\LT}$-algebra that is finitely generated and free as an $\Sect_{\LT}$-module.

We will now show that the functor corepresented by the underlying complete local ring
\[
\hom(\Sect_{\Gu^A},-) \colon \Comprings \to \text{Set}
\]
is naturally isomorphic to $\Hom(A, \Gu)$. 

A map of complete local rings $\Sect_{\Gu^A} \to R$ is equivalent to the choice of an $\Sect_{\LT}$-algebra structure on $R$, $j \colon \Sect_{\LT} \to R$, and a homomorphism $A \to j^*\Gu$. Using this data, we may form the deformation
\[
(A \to j^* \Gu, j/\m, \id_{(j/\m)^*\F}).
\]

Let $(f \colon A \to \G, i, \tau)$ be an object in the groupoid $\Hom(A, \Gu)(R, \m)$. By Theorem \ref{lubintate}, there is a map $j \colon \Sect_{\LT} \to R$ and a unique $\star$-isomorphism
\[
(\G, i, \tau) \cong (j^* \Gu, j/\m, \id_{(j/\m)^*\F})
\]
which includes the data of an isomorphism of formal groups $t \colon \G \cong j^*\Gu$. Thus the deformation $(f, i, \tau)$ is (necessarily uniquely) $\star$-isomorphic to the deformation
\[
(A \lra{f} \G \lra{t} j^*\Gu, j/\m, \id_{(j/\m)^*\F}).
\]
Now the data of $j$ and $A \to j^*\Gu$ is equivalent to the data of a map $\Sect_{\Gu^A} \to R$. These two constructions are clearly inverse to each other.
\end{proof}

Let $A$ be a finite abelian $p$-group. Level structures on formal groups were introduced by Drinfeld in \cite[Section 4]{drinfeldell1}. A level structure $f \colon A \to \G$ is a homomorphism from $A$ to $\G$ over $R$ with a property. There is a divisor on $\G$ associated to every homomorphism from $A$ to $\G$ and the idea is that a level structure is a homomorphism over $R$ for which the associated divisor is a subgroup scheme of $\G$. More explictly, fix a coordinate on $\G$, which provides an isomorphism $\Sect_{\G} \cong R\powser{x}$. For $f \colon A \to \G(R)$, let $g_f(x) = \Prod{a \in A} (x-f(a))$. The divisor associated to $f$ is $\Spf(R\powser{x}/(g_f(x)))$ and $f$ is a level structure if this is a subgroup scheme of $\G$ and the rank of $A$ is less than or equal to $n$ (the height). It is useful to think about level structures as injective maps, even though they are not necessarily injective in any sense.

We define
\[
\Level(A,\Gu) \colon \Comprings \to \Groupoids
\]
to be the functor that assigns to a complete local ring $(R,\m)$ the groupoid $\Level(A,\Gu)(R,\m)$ with objects triples $(l \colon A \to \G, i, \tau)$, where $l$ is a level structure, and morphisms $\star$-isomorphisms that commute with the level structures in the same way as in the definition of $\Hom(A,\Gu)(R,\m)$. Again, there is a faithful forgetful natural transformation $\Level(A,\Gu) \to \LT$. The proof of the next corollary is along the same lines as the proof of Corollary \ref{HomLT} with some extra work needed to show that the complete local ring is a domain.

\begin{cor} \cite[Proposition 4.3]{drinfeldell1}
The functor $\Level(A,\Gu)$ is corepresentable by a complete local domain $\Sect_{\Level(A,\Gu)}$ that is finitely generated and free as a module over $\Sect_{\LT}$.
\end{cor}

A particular case of this moduli problem will play an important role. Let $\lat = \Z_{p}^{\times n}$ and let $\qz = \lat^*$, the Pontryagin dual of $\lat$. The abelian group $\qz$ is abstractly isomorphic to $\QZ{\times n}$. We use the notation $\lat$ to indicate that we have fixed a ($p$-adic) lattice for which the Pontryagin dual $\qz$ is a ($p$-adic) torus. Let $\qzpk$ be the $p^k$-torsion in the torus and let 
\[
\lat_{p^k} = p^k\lat \subset \lat.
\]
By construction, there is a canonical isomorphism $\qzpk \cong (\latpk)^*$. Both $\qzpk$ and $\latpk$ are abstractly isomorphic to $(\Z/p^k)^{\times n}$. The functors $\Level(\qzpk, \Gu)$ are particularly important.

Finally we introduce one last moduli problem over $\LT$. Define 
\[
\Sub_{p^k}(\Gu) \colon \Comprings \to \Groupoids
\]
to be the functor that assigns to a complete local ring $(R,\m)$ the groupoid $\Sub_{p^k}(\Gu)(R,\m)$ with objects triples $(H \subset \G, i, \tau)$, where $H$ is a subgroup scheme of $\G$ of order $p^k$. Morphisms are $\star$-isomorphisms that send the chosen subgroup to the chosen subgroup. Just as with the other moduli problems, there is a faithful forgetful natural transformation $\Sub_{p^k}(\Gu) \to \LT$. Once again, the proof of the next corollary is along the same lines as the proof of Corollary \ref{HomLT} with some extra work needed to show that the complete local ring is finitely generated and free as an $\Sect_{\LT}$-module.

\begin{cor} \cite[Theorem 42]{subgroups} \label{ltsub}
The moduli problem $\Sub_{p^k}(\Gu)$ is corepresented by a complete local ring $\Sect_{\Sub_{p^k}(\Gu)}$ that is finitely generated and free as a module over $\Sect_{\LT}$.
\end{cor}

\section{Symmetries of certain moduli problems over Lubin-Tate space}

Besides being corepresentable, the moduli problems described in the last section are highly symmetric. We begin by describing the action of the stabilizer group on each of the moduli problems and then we will explain how to take the quotient of a deformation by a finite subgroup and how a certain limit of these moduli problems picks up extra symmetry. A reference that includes a concise exposition of much of the material in this section is \cite[Section 1]{carayolltt}.

The group $\Aut(\F / \ka)$ is known as the Morava stabilizer group. This group acts on the right on all of the moduli problems that we have described in a very simple way. Let $s \in \Aut(\F / \ka)$, then
\begin{equation} \label{stabaction}
s \cdot (\G, i, \tau)  = (\G, i, (i^*s) \tau ),
\end{equation}
where $i^*s \colon i^* \F \to i^* \F$ is the induced isomorphism. This is well-defined: if $\delta$ is a $\star$-isomorphism from $(\G, i , \tau)$ to $(\G', i, \tau')$ then $f$ is still a $\star$-isomorphism between $s \cdot (\G, i , \tau)$ and $s \cdot (\G', i, \tau')$ since
\[
\xymatrix{\pi^* \G \ar[r]^{\tau} \ar[d]^{\pi^*\delta} & i^* \F \ar[r]^-{i^*s} \ar[d]_{=} & i^* \F  \ar[d]_{=}   \\ \pi^* \G' \ar[r]^{\tau'} & i^* \F \ar[r]^-{i^*s} & i^* \F   }
\]
commutes. The actions of $\Aut(\F / \ka)$ on $\Hom(A,\Gu)$, $\Level(A,\Gu)$, and $\Sub_{p^k}(\Gu)$ are defined similarly.

There is a right action of $\Aut(A)$ on both $\Hom(A,\Gu)$ and $\Level(A,\Gu)$ given by precomposition and a left action of $\Aut(A)$ on both $\Hom(A,\Gu)$ and $\Level(A,\Gu)$ given by precomposition with the inverse. Since neither of these actions affects $\G$, $i$, or $\tau$, the corresponding actions of $\Aut(A)$ on $\Sect_{\Hom(A,\Gu)}$ and $\Sect_{\Level(A,\Gu)}$ are by $\Sect_{\LT}$-algebra maps. We mention the left action by precomposition with the inverse because that is the action that we will generalize and that will play the most important role in our constructions. Specializing to $A = \qzpk$, we see that there is a left action of $\Aut(\qzpk) \cong\GL_n(\Zp{k})$ on $\Level(\qzpk,\Gu)$ given by precomposition with the inverse. 

There is a forgetful functor $\Level(\qzpk,\Gu) \to \Level(\qz[p^{k-1}],\Gu)$ induced by the map sending a level structure $\qzpk \to \G$ to the level structure given by the composite $\qz[p^{k-1}] \subset \qzpk \to \G$. Let 
\[
\Level(\qz,\Gu) = \lim_{k} \Level(\qzpk,\Gu). 
\]
We will refer to the ring of functions $\Sect_{\Level(\qz,\Gu)} = \colim_k \Sect_{\Level(\qzpk,\Gu)}$ as the Drinfeld ring. The groupoid $\Level(\qz,\Gu)(R,\m)$ can be thought of as follows: It consists of tuples $(l \colon \qz \to \G, i, \tau)$ up to $\star$-isomorphism, where $l$ is a map from $\qz$ to $\G$ over $R$ such that, for all $k \geq 0$, the induced map $\qz[p^k] \to \G$ is a level structure. If $H \subset \qz$ is a subgroup of order $p^k$, then $H$ determines a map $\Level(\qz,\Gu) \to \Sub_{p^k}(\Gu)$ defined by
\[
(l, i, \tau) \mapsto (l(H) \subset \G, i, \tau).
\]

There is a left action of $\Aut(\qz)$ on $\Level(\qz,\Gu)$ given by precomposition with the inverse. An isogeny $\qz \to \qz$ is a homomorphism with finite kernel. It turns out that the action of $\Aut(\qz)$ on $\Level(\qz,\Gu)$ extends to an action of $\Isog(\qz)$, the monoid of endoisogenies of $\qz$, a much larger object! To describe this action, we first need to describe the quotient of a deformation by a finite subgroup.

Let $\sigma \colon \Spec(\ka) \to \Spec(\ka)$ be the Frobenius endomorphism on $\ka$ and let $\sigma_{\F} \colon \F \to \F$ be the Frobenius endomorphism on the formal group $\F$. Let $\sigma^* \F$ be the pullback of $\F$ along $\sigma$. The relative Frobenius is an isogeny of formal groups of degree $p$
\[
\Frob \colon \F \to \sigma^* \F.
\]
It is constructed using the universal property of the pullback:
\begin{equation} \label{frob}
\xymatrix{\F \ar@/_1pc/[ddr] \ar@/^1pc/[rrd]^{\sigma_{\F}} \ar@{-->}[dr]|{\Frob} & & \\ & \sigma^* \F \ar[r] \ar[d] & \F \ar[d] \\ & \Spec(\ka) \ar[r]^-{\sigma} & \Spec(\ka).}
\end{equation}
The formal group $\F$ lives over $\ka$, so there is a canonical map $\F \to \Spec(\ka)$. This fits together with the $\sigma_{\F}$ to produce the map $\Frob \colon \F \to \sigma^* \F$. Using the $k$th power of the Frobenius $\sigma^k$ and $\sigma_{\F}^{k}$ in the diagram above instead of $\sigma$ and $\sigma_{\F}$, we construct the $k$th relative Frobenius $\Frob^k \colon \F \to (\sigma^k)^* \F$. The kernel of the $k$th relative Frobenius is a subgroup scheme of $\F$ of order $p^k$. A formal group over a field of characteristic $p$ has a unique subgroup scheme of order $p^k$ for each $k \geq 0$. Thus $\ker(\Frob^k) \subset \F$ is the unique subgroup scheme of $\F$ of order $p^k$.

Let $(\G, i, \tau)$ be a deformation of $\F/\ka$ to a complete local ring $R$ and let $H \subset \G$ be a subgroup scheme of order $p^k$. The quotient of $\G$ by $H$ is the formal group 
\[
\G/H = \xymatrix{\text{Coeq}(\G \times H \ar@<.5ex>[r]^-{\text{act}} \ar@<-.5ex>[r]_-{\text{proj.}} & \G)}
\]
constructed as the coequalizer of the action map and the projection map as in \cite[Section 5]{subgroups}. There is a canonical way to extend the quotient formal group $\G/H$ to a deformation of $\F$. Let $q \colon \G \to \G/H$ be the quotient map.  Consider the following diagram of formal groups over $R/m$
\[
\xymatrix{\pi^*\G \ar[r]^{\tau} \ar[d]^{\pi^*q} & i^*\F  \ar[d]^{i^*\Frob^k} \\ \pi^*(\G/H)  \ar@{-->}[r]^{\tau/H} & i^* (\sigma^k)^* \F}
\]
Since both $\ker(\pi^*q) \subset \pi^*\G$ and $\ker(i^* \Frob^k \circ \tau) \subset \pi^*\G$ have the same order, we must have $\ker(\pi^*q) = \ker(i^* \Frob^k \circ \tau)$. Thus, by the first isomorphism theorem for formal groups, there is a unique isomorphism (the dashed arrow) making the diagram commutes. We will call this isomorphism $\tau/H$. Thus, given the deformation $(\G, i, \tau)$ over $R$ and $H \subset \G$, we may form the deformation
\[
(\G/H, i \circ \sigma^k, \tau/H).
\]

Finally, we will produce a left action of the monoid $\Isog(\qz)$ on $\Level(\qz,\Gu)$. First note that $\Isog(\qz)$ really is much larger than $\Aut(\qz)$. Pontryagin duality produces an anti-isomorphism between $\Aut(\qz)$ and $\GL_n(\Z_p)$ and an anti-isomorphism between $\Isog(\qz)$ and the monoid of $n \times n$ matrices with coefficients in $\Z_p$ and non-zero determinant. Let $H \subset \qz$ be a finite subgroup of order $p^k$, let $\phi_H \in \Isog(\qz)$ be an endoisogeny of $\qz$ with kernel $H$, and let $q_H \colon \qz \to \qz/H$ be the quotient map. Recall that if $l \colon \qzpk \to \G$ is a level structure and $H \subset \qzpk$, then $l(H) \subset \G$ is a subgroup scheme of order $|H|$. We will abuse notation and refer to $l(H) \subset \G$ as $H$. 

For a deformation equipped with level structure
\[
(l \colon \qz \to \G, i, \tau),
\]
we set 
\begin{equation} \label{isogaction}
\phi_H \cdot (l, i, \tau)    = (l/H \circ \psi_{H}^{-1}, i \circ \sigma^k, \tau/H),
\end{equation}
where $l/H \colon \qz/H \to \G/H$ is the induced level structure and $\psi_H$ is the unique isomorphism that makes the square
\begin{equation} \label{psidiagram}
\xymatrix{ \qz \ar[d]_{q_H} \ar[dr]^-{\phi_H} & \\ \qz/H \ar[r]^-{\psi_H}_-{\cong} & \qz}
\end{equation}
commute. We leave it to the reader to define the action of $\phi_H$ on $\star$-isomorphisms and to check that the resulting formula does in fact give an action.

\begin{prop}
The formula of Equation \eqref{isogaction} defines a left action of $\Isog(\qz)$ on $\Level(\qz,\Gu)$.
\end{prop}

Restricting the above action to the case that $H=e$, the trivial subgroup, we that $q_e$ is the identity map and that $\psi_e = \phi_e$ so that
\[
\phi_e \cdot (l,i,\tau) = (l \circ \phi_{e}^{-1}, i, \tau).
\]
This shows that the action of $\Isog(\qz)$ on $\Level(\qz,\Gu)$ extends the action of $\Aut(\qz)$ given by precomposition with the inverse. Since we have a left action of $\Isog(\qz)$ on $\Level(\qz,\Gu)$, we have a right action of $\Isog(\qz)$ on $\Sect_{\Level(\qz,\Gu)}$.


\section{Morava $E$-theory} \label{etheory}
In this section we construct Morava $E$-theory using the Landweber exact functor theorem, we calculate the $E$-cohomology of finite abelian groups, we describe the Goerss-Hopkins-Miller theorem, and we describe the resulting power operations on $E$-cohomology.


Morava $E$-theory is a cohomology theory that is built using the Landweber exact functor theorem using the universal formal group over the Lubin-Tate ring associated to the height $n$ formal group $\F/\ka$. Theorem \ref{lubintate} states that there is a non-canonical isomorphism
\[
\Sect_{\LT} \cong W(\ka)\powser{u_1, \ldots, u_{n-1}}.
\]
It follows from \cite[Section 6]{subgroups} that a coordinate can be chosen on the universal deformation formal group $\G_u$ such that the associated formal group law has the property that
\[
[p]_{\G_u}(x) = u_ix^{p^i} \mod (p, u_1, \ldots u_{i-1}, x^{p^i+1}).
\]
Since $(p,u_1, \ldots, u_{n-1})$ is a regular sequence in $W(\ka)\powser{u_1, \ldots, u_{n-1}}$, this property allows the Landweber exact functor theorem to be applied producing a homology theory $E_*(-) = E(\F,\ka)_*(-)$, called Morava $E$-theory, by the formula
\[
E_*(X) = \Sect_{\LT} \otimes_{MUP_0} MUP_*(X),
\] 
where $MUP$ is $2$-periodic complex cobordism. Good introductions to the Landweber exact functor theorem can be found in \cite{coctalos} and \cite[Lecture 16]{luriechromatic}.

\begin{example}
When $(\F, \ka) = (\hat{\G}_m, \F_p)$, the formal multiplicative group over the field with $p$ elements, then $\G_u = \hat{\G}_m$ over $W(\ka) = \Z_p$. The resulting cohomology theory is precisely $p$-adic $K$-theory.
\end{example}

There is a cohomology theory associated to this homology theory. On spaces equivalent to a finite CW-complex, it satisfies
\[
E^*(X) = \Sect_{\LT} \otimes_{MUP_0} MUP^*(X).
\]
From now on, we will always write $E^0$ rather than $\Sect_{\LT}$ for the Lubin-Tate ring. We will primarily be interested in spaces of the form $BG$ for $G$ a finite group and these are almost always not equivalent to a finite CW complex. To get a grip on the $E$-cohomology of these large spaces, we boot strap from the calculation of $E^*(BS^1)$ using the Atiyah-Hirzebruch spectral sequence, the Milnor sequence, and the Mittag-Leffler condition to control $\lim^1$ as described very clearly in \cite[Chapter 2.2]{millernotes}.


\begin{prop} \cite[Theorem 2.3]{millernotes} \label{S1}
There is an isomorphism of $E^*$-algebras
\[
E^*(BS^1) \cong E^*\powser{x}.
\]
\end{prop}


Hopkins, Kuhn, and Ravenel observe in \cite[Lemma 5.7]{hkr} that the Gysin sequence associated to the circle bundle $S^1 \to BC_{p^k} \to BS^1$ can be used to calculate $E^*(BC_{p^k})$ from Proposition \ref{S1}.

\begin{prop}
The isomorphism of Proposition \ref{S1} induces an isomorphism of $E^*$-algebras
\[
E^*(BC_{p^k}) \cong E^*\powser{x}/([p^k]_{\G_u}(x)).
\]
\end{prop}
\begin{proof}
The Gysin sequence associated to the circle bundle $S^1 \to BC_{p^k} \to BS^1$ is the exact sequence
\[
\xymatrix{E^*\powser{x} \ar[rr]^{\times [p^k]_{\Gu}(x)} & & E^*\powser{x} \ar[r] &  E^*(BC_{p^k}),}
\]
where $[p^k]_{\Gu}(x)$ is in degree $-2$. Since $E^*\powser{x}$ is a domain, multiplication by $[p^k]_{\Gu}(x)$ is injective and the result follows.
\end{proof}

As we noted in the proof of Corollary \ref{HomLT}, the Weierstrass preparation theorem implies that $E^*\powser{x}/([p^k]_{\G_u}(x))$ is a finitely generated free module over $E^*$. Since $E^*(BC_{p^k})$ is a finitely generated free module, we can use the Kunneth isomorphism to boot strap our calculation to finite abelian groups. 
\begin{prop} \label{abcomputation}
Let $A \cong \prod_{i=1}^{m} C_{p^{k_i}}$, then there is an isomorphism of $E^*$-algebras
\[
E^*(BA) \cong E^*\powser{x_1, \ldots, x_m}/([p^{k_1}]_{\G_u}(x), \ldots, [p^{k_m}]_{\G_u}(x)).
\]
\end{prop}

In this article, we will restrict our attention to the $0$th cohomology ring $E^0(X)$. Since $E$ is even periodic, this commutative ring contains plenty of information. The isomorphism of Proposition \ref{abcomputation} induces an isomorphism of $E^0$-algebras
\[
E^0(BA) \cong E^0\powser{x_1, \ldots, x_m}/([p^{k_1}]_{\G_u}(x), \ldots, [p^{k_m}]_{\G_u}(x)).
\]
Given the proof of Corollary \ref{HomLT}, the careful reader may already recognize a close relationship between $E^0(BA)$ and $\Hom(A, \Gu)$.


Now that we have calculated the $E$-cohomology of finite abelian groups we describe some other properties of Morava $E$-theory. The most important theorem regarding Morava $E$-theory is due to Goerss, Hopkins, and Miller \cite{structuredmoravae} and we will endeavor to describe it now. Given a pair $(\F, \ka)$, we have produced a homology theory $E(\F, \ka)_*(-)$. Let $\FG$ denote the category with objects pairs $(\F, \ka)$ of a perfect field of characteristic $p$ and a finite height formal group over $\ka$. A morphism from $(\F, \ka)$ to $(\F', \ka')$ is given by a map $j \colon \ka \to \ka'$ and an isomorphism of formal group laws $j^*\F \to \F'$. This category is equivalent to the category called $\FGL$ in \cite{structuredmoravae}. Goerss, Hopkins, and Miller produce a fully faithful functor
\[
\xymatrix{\FG \ar[r]^-{E(-,-)} & E_{\infty}\text{-ring spectra}}
\]
such that the underlying homology theory of the $E_{\infty}$-ring spectrum $E(\F, \ka)$ is $E(\F, \ka)_*(-)$. Moreover, they show that the space of $E_{\infty}$-ring spectra $R$ with $E(\F, \ka)_*R$ an algebra object in graded $E(\F, \ka)_*E(\F, \ka)$-comodules isomorphic to $E(\F, \ka)_*E(\F, \ka)$ is equivalent to $B\Aut_{\FG}(\F,\ka)$. Thus the group of automorphisms of $E(\F,\ka)$ as an $E_{\infty}$-ring spectrum is the extended stabilizer group $\Aut_{\FG}(\F,\ka) \cong \Aut(\F / \ka) \rtimes \Gal(\ka,\F_p)$. It follows that $\Aut(\F / \ka)$ acts on $E(\F,\ka)^*(X)$ by ring maps for any space $X$. 




In the homotopy category, the $E_{\infty}$-ring structure on Morava $E$-theory gives rise to power operations. For any $E_{\infty}$-ring spectrum, these are constructed as follows: Let $X$ be a space and let $\Sigma^{\infty}_{+}X \to E$ be a an element of $E^0(X)$. Applying the extended powers functor $(E\Sigma_m)_+ \wedge_{\Sigma_m} (-)^{\wedge m}$ to both sides produces a map of spectra
\[
(E\Sigma_m)_+ \wedge_{\Sigma_m} (\Sigma^{\infty}_{+}X)^{\wedge m} \simeq \Sigma^{\infty}_{+}(E\Sigma_m \times_{\Sigma_m} X^{\times m}) \to (E\Sigma_m)_+ \wedge_{\Sigma_m} E^{\wedge m}.
\] 
Now the $E_{\infty}$-ring structure on $E$ gives us a map 
\[
(E\Sigma_m)_+ \wedge_{\Sigma_m} E^{\wedge m} \to E.
\]
Composing these two maps produces the $m$th total power operation
\[
\Po_m \colon E^0(X) \to E^0(E\Sigma_m \times_{\Sigma_m} X^{\times m}),
\]
a multiplicative, non-additive map. We will solely be concerned with the case $X = BG$, when $G$ is a finite group. There is an equivalence
\[
E\Sigma_m \times_{\Sigma_m} BG^{\times m} \simeq BG \wr \Sigma_m,
\]
where $G \wr \Sigma_m = G^{\times m} \rtimes \Sigma_m$ is the wreath product. In this case, the total power operation takes the form
\begin{equation} \label{totalpower}
\Po_m \colon E^0(BG) \to E^0(BG \wr \Sigma_m).
\end{equation}
We will make use of variants of the total power operation which we will introduce after a brief interlude regarding transfers.

As $E^*(-)$ is a cohomology theory, it has a theory of transfer maps for finite covers. For $H \subset G$, the map $BH \to BG$ is a finite cover and thus there is a transfer map in $E$-cohomology of the form
\[
\Tr_E \colon E^0(BH) \to E^0(BG).
\]
This is a map of $E^0(BG)$-modules for the $E^0(BG)$-module structure on the domain from the restriction map. We will see that these transfers play an important role in understanding the additive properties of power operations and that they interact nicely with the character maps of Hopkins, Kuhn, and Ravenel. In particular, these transfers can be used to define several important ideals. Assume $i,j>0$ and that $i+j = m$, then there is a transfer map
\[
\Tr_E \colon E^0(BG \wr (\Sigma_i \times \Sigma_j)) \to E^0(BG \wr \Sigma_m)
\]
induced by the inclusion $G \wr (\Sigma_i \times \Sigma_j) \subset G \wr \Sigma_m$. Summing over these maps as $i$ and $j$ vary gives a map of $E^0(BG \wr \Sigma_m)$-modules
\[
\Oplus{i,j}\Tr_E \colon \Oplus{i,j}E^0(BG \wr (\Sigma_i \times \Sigma_j)) \to E^0(BG \wr \Sigma_m).
\]
Define the ideal $J_{\tr} \subset E^0(BG \wr \Sigma_{m})$ by
\[
J_{\tr} = \im(\Oplus{i,j}\Tr_E).
\]
When $G$ is trivial, we will refer to this ideal as $I_{\tr} \subset E^0(B\Sigma_m)$.


Since $\Sigma_m$ acts trivially on the image of the diagonal $\Delta \colon G \to G^m$, we have a map of spaces $BG \times B\Sigma_m \to BG \wr \Sigma_m$. Restriction along this map gives us the power operation
\[
P_m \colon E^0(BG) \lra{\Po_m} E^0(BG \wr \Sigma_m) \to E^0(BG \times B \Sigma_m).
\] 
Since $E^0(B\Sigma_m)$ is a free $E^0$-module by \cite[Theorem 3.2]{etheorysym}, we have a Kunneth isomorphism
\[
E^0(BG \times B \Sigma_m) \cong E^0(BG) \otimes_{E^0} E^0(B\Sigma_m).
\]
We will also denote the resulting map by
\[
P_m \colon E^0(BG) \to E^0(BG) \otimes_{E^0} E^0(B\Sigma_m)
\]
and refer to it as the power operation as well. Both the total power operation $\Po_m$ and the power operation $P_m$ are multiplicative non-additive maps. Let $\iota \colon * \to B\Sigma_m$ be the inclusion of base point. The maps $\Po_m$ and $P_m$ are called power operations because the composite
\[
E^0(BG) \lra{P_m} E^0(BG) \otimes_{E^0} E^0(B\Sigma_m) \lra{\id \otimes \iota^*} E^0(BG) 
\]
is the $m$th power map $(-)^m \colon E^0(BG) \to E^0(BG)$. 


The failure of the total power operation to be additive is controlled by the ideal $J_{\tr}$. In other words, $J_{\tr} \subset E^0(BG \wr \Sigma_{m})$ is the smallest ideal such that the quotient
\begin{equation} \label{additivetotal}
\Po_m/J_{\tr} \colon E^0(BG) \to E^0(BG \wr \Sigma_m)/J_{\tr}
\end{equation}
is a map of commutative rings. The same relationship holds for the ideal $I_{\tr}$ and the power operation $P_m$, the quotient map
\[
P_m/I_{\tr} \colon E^0(BG) \to E^0(BG) \otimes_{E^0} E^0(B\Sigma_m)/I_{tr}
\]
is a map of commutative rings.

The total power operation and its variants satisfy several identities that hold for any $E_{\infty}$-ring spectrum. These are described in great detail in Chapters I and VIII of \cite{bmms}. For the purposes of this paper, we will need the following identity, which follows from \cite[VIII Proposition 1.1.(i)]{bmms}, that will come in handy for the induction arguments in Section \ref{secmain}: Let $i,j > 0$ such that $i+j = m$ and let $\Delta_{i,j} \colon \Sigma_i \times \Sigma_j \to \Sigma_m$ be the canonical map and also recall that $\Delta \colon G \to G \times G$ is the diagonal. The following diagram commutes
\begin{equation} \label{powerdiagram}
\xymatrix{E^0(BG) \ar[rr]^-{P_m} \ar[d]_-{P_i \times P_j} & & E^0(BG \times B\Sigma_m) \ar[d]^-{\Delta_{i,j}^{*}} \\ E^0(BG \times B\Sigma_i \times BG \times B\Sigma_j) \ar[rr]^-{\Delta^*} & & E^0(BG \times B\Sigma_i \times B\Sigma_j),}
\end{equation}
where $P_i \times P_j$ is the external product.
%
%
%
%

%

\section{Character theory} \label{sec:characters}

Since Morava $E$-theory is a well-behaved generalization of $p$-adic $K$-theory, one might wonder if some of the more geometric properties of $K$-theory can be extended to $E$-theory. The character map in classical representation theory is an injective map of commutative rings from the representation ring of $G$ to the ring of class functions on $G$ taking values in $\C$
\[
\chi \colon R(G) \to \Cl(G,\C).
\]
It is the map taking a representation $\rho \colon G \to \GL_n(\C)$ to the class function sending $g$ to $\Tr(\rho(g))$, the trace of $\rho(g)$. It has the important property that 
\[
\C \otimes \chi \colon \C \otimes_{\Z} R(G) \to \Cl(G,\C)
\]
is an isomorphism. 

In \cite{hkr}, Hopkins, Kuhn, and Ravenel tackled the question of generalizing the character map to Morava $E$-theory. They did not have geometric cocycles (such as $G$-representations) available to them, but they did have a generalization of the representation ring of $G$: $E^0(BG)$. Following the construction of the classical character map, they produce an $E^0$-algebra $C_0$ that plays the role of $\C$ above, a ring of ``generalized" class functions $\Cl_n(G,C_0)$, and a character map
\[
\chi \colon E^0(BG) \to \Cl_n(G,C_0)
\]
with the property that the induced map
\[
C_0 \otimes_{E^0} E^0(BG) \to \Cl_n(G,C_0)
\]
is an isomorphism. In this section we will give a short exposition of their work on character theory and explain the relationship between their character map and the action of the stablizer group on $E^0(BG)$. A longer exposition on character theory can be found in \cite[Appendix A.1]{eric}, written by the author, and of course, the best source is the original \cite{hkr}.

We begin by describing Proposition 5.12 in \cite{hkr}, which draws a connection between the group cohomology calculations of Section \ref{etheory} and the moduli problem $\Hom(A,\Gu)$ of Section \ref{LT}. Recall that $A^* = \hom(A,S^1)$ is the Pontryagin dual of $A$.

\begin{prop} \label{eabhom}
Let $A$ be a finite abelian group. There is a canonical isomorphism of $E^0$-algebras
\[
E^0(BA) \cong \Sect_{\Hom(A^*,\Gu)}
\]
compatible with the action of the stabilizer group.
\end{prop}
\begin{proof}
Recall from the proof of Corollary \ref{HomLT} that $\Sect_{\Hom(A^*,\Gu)} = \Sect_{\Gu^{A^*}}$. Thus, given a complete local $E^0$-algebra $R$, we need to produce a canonical isomorphism 
\[
\hom_{E^0\text{-alg}}(E^0(BA), R) \cong \hom(A^*,\Gu(R)).
\]
Let $f \colon A \to S^1$ be an element of $A^*$ and let $E^0(BA) \to R$ be a map of $E^0$-algebras. Applying $E$-cohomology to $f$ gives a map $E^0(BS^1) \to E^0(BA)$, composing this with the map $E^0(BA) \to R$ gives us a map $E^0(BS^1) \to R$ which is equivalent to a point in $\Gu(R)$. The trick to building a map backwards is to notice that a generating set in $A^*$ determines an injection from $A$ into a finite dimensional torus.

Now we will sketch why the isomorphism respects the action of $\Aut(\F/\ka)$. The point is that we can reduce to the case $A = S^1$ and the stabilizer group acts on $\Gu$ (purposefully not over $E^0 = \Sect_{\LT}$) viewed as a functor from complete local rings to groupoids. This action can be seen in the construction of $E^*(-)$ as a cohomology theory by noting that the Landweber exact functor theorem can be upgraded in this situation to a functor from $\FGL$ to cohomology theories. 

We view $\Gu$ as a functor from complete local rings to groupoids in the usual way by sending a complete local ring $(R,\m)$ to the groupoid of deformations equipped with a marked point $(\Z \to \G, i, \tau)$ and $\star$-isomorphisms compatible with the map from $\Z$. The stabilizer group acts in the usual way. There is another way to describe this action, though. Associated to $s \in \Aut(\F,\ka)$ is the deformation
\[
(\Gu, \id_{\ka}, s).
\]
This determines an isomorphism of commutative rings $f_s \colon E^0 \to E^0$ such that $\pi^*(f_s)= \id_{\ka}$ and by Theorem \ref{lubintate} there is a unique isomorphism 
\[
g_s \colon \Gu \lra{\cong} f_{s}^*\Gu
\]
such that $\pi^*(g_s) = s$. Now we define the action of $s \in \Aut(\F,\ka)$ on $\Gu$ to be the composite along the top
\[
\xymatrix{\Gu \ar[r]^-{g_s} & f_{s}^* \Gu \ar[r] \ar[d] & \Gu \ar[d] \\ & \LT \ar[r]^{\Spf(f_s)} & \LT,}
\]
where the square is the pullback square defining $f_{s}^*\Gu$. We connect to topology as follows: it suffices to prove the result for $A = S^1$ in which case we must show that the $\Aut(\F,\ka)$-action on $E^0(BS^1)$ and the $\Aut(\F, \ka)$ action on $\Gu$ described above agree. This follows from Rezk's functorial construction of the Morava $E$-theories (as cohomology theories) using the Landweber exact functor theorem in Section 6.7 of \cite{Rezk-Notes}.

\end{proof}

The ring $C_0$ is defined to be the rationalization of the Drinfeld ring $\Sect_{\Level(\qz,\Gu)}$. It has the property that for all $k \geq 0$, there is a canonical isomorphism of group schemes
\[
C_0 \otimes_{E^0} \Gu[p^k] \cong \qzpk.
\]
In fact, this can be taken as the defining property of $C_0$, but we will not need this fact in this document. Since $C_0$ is the rationalization of the Drinfeld ring, it comes equipped with a right action of $\Isog(\qz)$ by ring maps and an action of $\Aut(\qz)$ by $\Q \otimes E_0$-algebra maps. In fact, $C_0$ is an $\Aut(\qz)$-Galois extension of $\Q \otimes E^0$ so there is an isomorphism
\[
C_{0}^{\Aut(\qz)} \cong \Q \otimes E^0.
\]

Let $\hom(\lat, G)$ be the set of continuous homomorphisms from $\lat$ to $G$. This set is in bijective correspondence with the set of $n$-tuples of pairwise commuting $p$-power order elements of $G$. There is an action of $G$ on this set by conjugation. We define 
\[
\Cl_n(G,C_0) = \Prod{\hom(\lat, G)_{/\sim}} C_0, 
\]
where $\hom(\lat, G)_{/\sim}$ is the set of conjugacy classes of maps from $\lat$ to $G$. Thus $\Cl_n(G,C_0)$ is the ring of conjugation invariant functions on $\hom(\lat,G)$ taking values in $C_0$.

The character map is defined as follows: Given a conjugacy class $\al \colon \lat \to G$, there exists a $k \in \N$ such that $\al$ factors through $\latpk$. Applying the classifying space functor to this map gives a map of spaces
\[
B\latpk \to BG.
\]
Applying $E$-cohomology to this map gives a map of $E^0$-algebras
\[
\al^* \colon E^0(BG) \to E^0(B\latpk).
\]
By Proposition \ref{eabhom}, the codomain is canonically isomorphic to $\Sect_{\Hom(\qzpk, \Gu)}$, which is the ring corepresenting $\Hom(\qzpk,\Gu)$. There are canonical forgetful maps of moduli problems $\Level(\qz,\Gu) \to \Level(\qzpk,\Gu) \to \Hom(\qzpk,\Gu)$ over $\LT$ that induce maps of $E_0$-algebras
\[
\Sect_{\Hom(\qzpk, \Gu)} \to \Sect_{\Level(\qzpk,\Gu)} \to \Sect_{\Level(\qz,\Gu)} \to C_0.
\]
Composing these gives an $E^0$-algebra map
\begin{align} \label{charactermap}
\chi_{[\al]} \colon E^0(BG) \lra{\al^*} E^0(B\latpk) \cong \Sect_{\Hom(\qzpk, \Gu)} \to \Sect_{\Level(\qzpk,\Gu)} \to \Sect_{\Level(\qz,\Gu)} \to C_0.
\end{align}
Putting these together for all of the conjugacy classes in $\hom(\lat, G)$ gives the character map
\[
\chi \colon E^0(BG) \to \Cl_n(G,C_0).
\]

\begin{example} \label{characterab}
Let $A$ be a finite abelian group and let $\al \colon \lat \to A$ be a map. In this case, $\chi_{[\al]}$ admits an interpretation completely in terms of the moduli problems in Section \ref{LT}. Equation \eqref{charactermap} implies that we may view $\chi_{[\al]}$ as landing in $\Sect_{\Level(\qz,\Gu)}$. Proposition \ref{eabhom} implies that the domain is canonically isomorphic to $\Sect_{\Hom(A^*,\Gu)}$. Putting these observations together, $\chi_{[\al]}$ gives us a map
\[
\Sect_{\Hom(A^*,\Gu)} \to \Sect_{\Level(\qz,\Gu)}
\]
or on the level of moduli problems, a map
\[
\Level(\qz,\Gu) \to \Hom(A^*,\Gu).
\]
Unwrapping the definition of $\chi_{[\al]}$, when this map is applied to a complete local ring $(R,m)$, it sends a deformation with level structure
\[
(l \colon \qz \to \G, i, \tau) 
\]
to the deformation equipped with homomorphism
\[
(A^* \lra{\al^*} \qz \lra{l} \G, i, \tau).
\]
\end{example}

There is a left action of $\Aut(\qz)$ on $\hom(\lat, G)_{/\sim}$ given by precomposition with the Pontryagin dual. Combining this with the right action of $\Aut(\qz)$ on $C_0$, there is a right action on $C_0$-valued function on $\hom(\lat, G)_{/\sim}$ (ie. $\Cl_n(G,C_0)$). Explicitly this action is defined as follows: Let $\phi \in \Aut(\qz)$, let $f \in \Cl_n(G,C_0)$, and let $[\al] \in \hom(\lat, G)_{/\sim}$ then
\[
(f \phi)([\al]) = (f([\al \phi^*])) \phi.
\]
To see that this is an action note that, for $\phi, \tau \in \Aut(\qz)$,
\[
((f \phi) \tau)([\al]) = (f \phi)([\al \tau^*])\tau = f([\al \tau^* \phi^*]) \phi \tau = f([\al (\phi \tau)^*]) \phi \tau.
\]

It turns out that the base change of the character map $\chi$ to $C_0$
\[
C_0 \otimes \chi \colon C_0 \otimes_{E^0} E^0(BG) \to \Cl_n(G,C_0)
\]
is equivariant with respect to the right $\Aut(\qz)$-action on the source given by the action of $\Aut(\qz)$ on the left tensor factor and the right $\Aut(\qz)$-action on the target given above.

\begin{thm} \cite[Theorem C]{hkr} \label{thmc}
The character map induces an isomorphism
\[
C_0 \otimes_{E^0} E^0(BG) \lra{\cong} \Cl_n(G,C_0)
\]
and taking $\Aut(\qz)$-fixed points gives an isomorphism
\[
\Q \otimes E^0(BG) \lra{\cong} \Cl_n(G,C_0)^{\Aut(\qz)}.
\]
\end{thm}

The primary goal of this document is to give an exposition of the relationship between the power operations of Section \ref{etheory} and the character map above. However, we are already in the position to explain the relationship between the stabilizer group action on $E^0(BG)$ and the character map, so we will do that now.

Since the stabilizer group $\Aut(\F / \ka)$ acts on $E$ by $E_{\infty}$-ring maps, it acts on the function spectrum $E^{X}$ by $E_{\infty}$-ring maps for any space $X$ and given a map of spaces $X \to Y$, the induced map of $E_{\infty}$-ring spectra
\[
E^Y \to E^X
\]
is $\Aut(\F / \ka)$-equivariant. Recall from Section \ref{LT} that $\Aut(\F / \ka)$ also acts on $\Sect_{\Hom(\qzpk,\Gu)}$, $\Sect_{\Level(\qzpk,\Gu)}$ and $\Sect_{\Level(\qz,\Gu)}$ by commutative ring maps and thus there is also an action of $\Aut(\F / \ka)$ on $C_0$. Proposition \ref{eabhom} implies that the canonical isomorphism
\[
E^0(B\latpk) \cong \Sect_{\Hom(\qzpk,\Gu)}
\]
is $\Aut(\F / \ka)$-equivariant. Putting all of this together, we see that the composite of Equation \ref{charactermap}
\[
E^0(BG) \to E^0(B\latpk) \cong \Sect_{\Hom(\qzpk,\Gu)} \to \Sect_{\Level(\qzpk,\Gu)} \to \Sect_{\Level(\qz,\Gu)} \to C_0
\]
is $\Aut(\F / \ka)$-equivariant. Since the character map is a product of maps of this form, we may conclude the following proposition. 

\begin{prop}
The Hopkins-Kuhn-Ravenel character map
\[
\chi \colon E^0(BG) \to \Cl_n(G,C_0)
\]
is $\Aut(\F / \ka)$-equivariant for the canonical action of $\Aut(\F / \ka)$ on $E^0(BG)$ and the diagonal action of $\Aut(\F / \ka)$ on $\Cl_n(G,C_0)$. 
\end{prop}
%
%

\section{Transfers and conjugacy classes of tuples in symmetric groups}
In keeping with the theme of this article, in this section we will describe the relationship between transfer maps for Morava $E$-theory and the character map. In Theorem D of \cite{hkr}, Hopkins, Kuhn, and Ravenel show that a surprisingly simple function between generalized class functions, inspired by the formula for induction in representation theory, is compatible with the transfer map for Morava $E$-theory. 


To describe Theorem D of \cite{hkr}, we need a bit of group-theoretic set-up. Let $H \subset G$ be a subgroup and let $[\al] \in \hom(\lat, G)_{/\sim}$. The set of cosets $G/H$ has a left action of $G$ and we can take the fixed points for the action of the image of $\al$ on $G/H$ to get the set $(G/H)^{\im \al}$. A key property of these cosets is that if $gH \in (G/H)^{\im \al}$, then $\im(g^{-1} \al g) \subset H$. 


Mimicking the formula for the transfer in representation theory, Hopkins, Kuhn, and Ravenel define a transfer map on generalized class functions,
\[
\Tr_{C_0} \colon \Cl_n(H, C_0) \to \Cl_n(G, C_0),
\]
for $H \subset G$ by the formula
\begin{equation} \label{transfereqn}
\Tr_{C_0}(f)([\al]) = \sum_{gH \in (G/H)^{\im \al}} f([g^{-1} \al g]) 
\end{equation}
for a conjugacy class $[\al] \in \hom(\lat, G)_{/\sim}$ and a generalized class function $f \in \Cl_n(H,C_0)$. Note that the formula is independent of the implicit choice of system of representatives of $G/H$.


\begin{thm} \cite[Theorem D]{hkr} \label{thmd} For any finite group $G$ and subgroup $H \subset G$, there is a commutative diagram
\[
\xymatrix{E^0(BH) \ar[r]^{\Tr_{E}} \ar[d]_{\chi} & E^0(BG) \ar[d]^{\chi} \\ \Cl_n(H,C_0) \ar[r]^{\Tr_{C_0}} & \Cl_n(G,C_0).}
\]
\end{thm}

We need one lemma, which we leave to the reader, that will help us compute with the transfer map $\Tr_{C_0}$.

\begin{lemma} \label{lemH}
Let $H \subset G$ be a subgroup and let $\al \colon \lat \to G$. If $g \in G$ has the property that $\im g^{-1}\al g \subset H$, then $gH \in (G/H)^{\im \al}$.
\end{lemma}

Applying Equation \eqref{transfereqn}, we may conclude from Lemma \ref{lemH} that if $\al \colon \lat \to G$ has the property that $\im g^{-1}\al g \subset H$ and if $1_{[g^{-1}\al g]} \in \Cl_n(H,C_0)$ is the class function with value $1$ on $[g^{-1}\al g]$ and $0$ elsewhere, then
\[
\Tr_{C_0}(1_{[g^{-1}\al g]})([\al]) > 0.
\] 
Also, if $[\beta] \neq [\al] \in \hom(\lat,G)_{/ \sim}$, then 
\[
\Tr_{C_0}(1_{[g^{-1}\al g]})([\beta]) = 0.
\]
Otherwise, $\beta$ would be conjugate to $g^{-1}\al g$, which is conjugate to $\al$.

Since the additive power operations are ring maps, they are easier to understand than the power operations. To study the additive power operations, we will understand the analogue of the ideal $I_{\tr} \subset E^0(B\Sigma_{p^k})$ in the ring of generalized class functions $\Cl_n(\Sigma_{p^k}, C_0)$. 

\begin{definition}
Let 
\[
\Sum_m(\qz) = \{\oplus_i H_i| H_i \subset \qz, \sum_{i} |H_i| = m\}
\]
be the set of formal sums of subgroups $H_i \subset \qz$ such that the sum of the orders is $m$.
\end{definition}

\begin{prop} \label{conjiso} 
There is a canonical bijection of sets
\[
\hom(\lat, \Sigma_m)_{/\sim} \cong \Sum_m(\qz).
\]
\end{prop}
\begin{proof}
Since $\Sigma_m = \Aut(\um)$, where $\um$ is a fixed $m$-element set, a map $\al \colon \lat \to \Sigma_m$ gives $\um$ the structure of an $\lat$-set. Maps $\al \colon \lat \to \Sigma_m$ and $\beta \colon \lat \to \Sigma_m$ are conjugate if and only if the corresponding $\lat$-sets are isomorphic. Thus $\hom(\lat, \Sigma_m)_{/\sim}$ is in bijective correspondence with the set of isomorphism classes of $\lat$-sets of size $m$.

Given $\al$, decompose $\um = \coprod_i \um_i$, where $\um_i$ is a transitive $\lat$-set and let $\lat_{i}$ be the stabilizer of any point in $\um_i$. Since $\lat$ is abelian, the stabilizer does not depend on the choice of point. We define $H_i = (\lat/\lat_i)^* \subset \lat^* = \qz$. This construction only depends on the conjugacy class of $\al$, so send $[\al]$ to $\oplus_i H_i$. We leave it to the reader to check that this map is a bijection.
\end{proof}

Let $\hom(\lat,\Sigma_m)^{\trans}$ be the set of transitive homomorphisms $\lat \to \Sigma_m$. Note that $\lat \to \Sigma_m$ is transitive if and only if the image is a transitive abelian subgroup of $\Sigma_m$. Since $\lat = \Z_{p}^n$, this can only occur when $m$ is a power of $p$.

\begin{cor} \label{subiso}
There is a canonical bijection
\[
\hom(\lat, \Sigma_{p^k})^{\trans}_{/\sim} \cong \Sub_{p^k}(\qz),
\]
where $\Sub_{p^k}(\qz)$ is the set of subgroups $H \subset \qz$ of order $p^k$.
\end{cor}
\begin{proof}
Follows immediately from the construction of the bijection of Proposition \ref{conjiso}.
\end{proof}

The next lemma will be useful when we try to understand the relation between additive power operations and character theory. Recall the inclusion $\Delta_{i,j} \colon \Sigma_i \times \Sigma_j \hookrightarrow \Sigma_m$, where $i+j = m$ and $i,j>0$. Proposition \ref{conjiso} gives an isomorphism
\[
\hom(\lat, \Sigma_i \times \Sigma_j)_{/\sim} \cong \Sum_i(\qz) \times \Sum_j(\qz).
\]
We leave the proof of the next lemma to the reader.

\begin{lemma} \label{summap}
Assume $i,j>0$ and $i+j = m$, then the map
\[
\Sum_i(\qz, G) \times \Sum_j(\qz,G) \to \Sum_m(\qz,G)
\]
induced by $\Delta_{i,j}$ sends a pair of sums of subgroups
\[
(\Oplus{l} H_l, \Oplus{l} K_l) 
\]
to the sum
\[
\Oplus{l} H_l \oplus \Oplus{l} K_l.
\]
\end{lemma}

\begin{cor}\label{padicsum}
Let $\sum_j a_j p^j$ be the $p$-adic expansion of $m$, then the inclusion
$\prod_{j} \Sigma_{p^j}^{\times a_j} \subseteq \Sigma_m$
induces a surjection
\[
\Prod{j} \Sum_{p^j}(\qz)^{\times a_j} \twoheadrightarrow \Sum_m(\qz).
\]
\end{cor}
\begin{proof}
Follows from Lemma \ref{summap}.
\end{proof}

Finally, we will use this lemma to give a description $\Cl_n(\Sigma_m, C_0)/I_{\tr}$, where $I_{\tr} \subset \Cl_n(\Sigma_m, C_0)$ is the image of the sum of the transfer maps 
\[
\Oplus{i,j}\Tr_{C_0} \colon \Cl_n(\Sigma_i \times \Sigma_j, C_0) \to \Cl_n(\Sigma_m, C_0)
\]
along the inclusions $\Delta_{i,j}$ for all $i,j>1$ with $i+j = m$.

\begin{prop}
There is a canonical isomorphism
\[
\Cl_n(\Sigma_m,C_0)/I_{\tr} \cong \Prod{\Sub_m(\qz)} C_0.
\]
\end{prop}
\begin{proof}
Proposition \ref{conjiso} implies that we have a canonical isomorphism
\[
\Cl_n(\Sigma_m,C_0) \cong \Prod{\Sum_m(\qz)} C_0.
\]
We will show that $I_{\tr} \subset \Cl_n(\Sigma_m,C_0)$ consists of the functions supported on the set of sums of subgroups with more than one summand.

First we will show that functions with support not contained in the set of sums with more than one summand cannot be hit by the transfer map. The definition of the transfer $\Tr_{C_0}$ is as a sum over elements that conjugate a conjugacy class in $[\lat \to G]$ into the subgroup $H$. In our case $G = \Sigma_{m}$ and $H = \Sigma_i \times \Sigma_j$. If $\lat \to \Sigma_m$ is transitive then every conjugate of the map is also transitive and so no conjugate can land inside $\Sigma_i \times \Sigma_j$.

Now let $\al \colon \lat \to \Sigma_m$ be a non-transitive map. Thus $[\al]$ corresponds to a sum of subgroups with more than one summand. Therefore there exists $i,j > 0$ with $i+j = m$ such that $\al$ is conjugate to a map $[\beta]$ that lands in $\Sigma_i \times \Sigma_j$. The discussion after Lemma \ref{lemH} implies that $\Tr_{C_0}(1_{[\beta]})$ is concentrated on $[\al]$ and there it is a non-zero natural number (which is invertible in $C_0$). Thus $I_{\tr}$ contains the factor of $\Cl_n(\Sigma_m,C_0)$ corresponding to $[\al]$.
\end{proof}


By Theorem \ref{thmd}, there is a commutative diagram
\[
\xymatrix{\Oplus{i,j}E^0(B\Sigma_i \times B\Sigma_j)) \ar[r]^-{\oplus \Tr_E} \ar[d] & E^0(B\Sigma_m) \ar[d] \\ \Oplus{i,j}\Cl_n(\Sigma_i \times \Sigma_j, C_0) \ar[r]^-{\oplus \Tr_{C_0}} & \Cl_n(\Sigma_m, C_0),}
\]
where the sum is over all $i,j>0$ such that $i+j = m$,
Thus, we have an induced map of $E^0$-algebras
\[
\chi \colon E^0(B\Sigma_m)/I_{\tr} \to \Cl_n(\Sigma_m,C_0)/I_{\tr} \cong \Prod{\Sub_{p^k}(\qz)} C_0
\]
that we will continue to call $\chi$. Theorem \ref{thmc} above implies that this map gives rise to an isomorphism
\[
C_0 \otimes_{E^0} E^0(B\Sigma_m)/I_{\tr} \lra{\cong} \Cl_n(\Sigma_m,C_0)/I_{\tr}.
\]
Assume that $H \in \Sub_{p^k}(\qz)$ corresponds to the transitive conjugacy class $[\beta \colon \lat \to \Sigma_{p^k}]$. We define $\chi_H$ to be the composite
\begin{equation} \label{chih}
\chi_H \colon E^0(B\Sigma_m)/I_{\tr} \to \Cl_n(\Sigma_m,C_0)/I_{\tr} \lra{\pi_H} C_0.
\end{equation}
Since the quotient map $\Cl_n(\Sigma_m,C_0) \to \Cl_n(\Sigma_m,C_0)/I_{\tr}$ is just a projection, the maps $\chi_H$ and $\chi_{[\beta]}$ are related by the following commutative diagram:
\[
\xymatrix{E^0(B\Sigma_{p^k}) \ar[r] \ar[dr]_{\chi_{[\beta]}} & E^0(B\Sigma_m)/I_{\tr} \ar[d]^{\chi_H} \\ & C_0.}
\]

\section{A theorem of Ando-Hopkins-Strickland}

An important ingredient in understanding the relationship between power operations and character theory is a result of Ando, Hopkins, and Strickland that gives an algebro-geometric interpretation of a special case of the power operation in terms of Lubin-Tate theory. Their result indicates that there is a connection between power operations and the $\Isog(\qz)$-action on $C_0$. 

One starting point for making the connection between algebraic geometry and power operations is Strickland's theorem \cite{etheorysym} which gives an algebro-geometric interpretation of $E^0(B\Sigma_{p^k})/I_{\tr}$. 

\begin{thm} \cite[Theorem 9.2]{etheorysym} \label{strickland}
There is a canonical isomorphism of $E^0$-algebras
\[
E^0(B\Sigma_{p^k})/I_{\tr} \cong \Sect_{\Sub_{p^k}(\Gu)}.
\]
\end{thm}
\begin{proof}[idea of the proof]
It is worth understanding the origin of this isomorphism. Let 
\[
\Div_{p^k}(\Gu) \colon \Comprings_{E^0/} \to \text{Set}
\]
be the functor that assigns to $j \colon E^0 \to R$ the set of effective divisors on $j^*\Gu$. These divisors are just subschemes of $j^*\Gu$ of the form
\[
\Spf(R\powser{x}/(a_0 +a_{1}x + \ldots + a_{p^k-1}x^{p^k-1}+x^{p^k})),
\]
where $a_i$ is an element of the maximal ideal of $R$. Since a subgroup scheme has an underlying divisor, $\Sub_{p^k}(\Gu)$ is a closed subscheme of $\Div_{p^k}(\Gu)$. Proposition 8.31 in \cite{fsfg} gives a canonical isomorphism of formal schemes
\[
\Spf E^0(BU(p^k)) \cong \Div_{p^k}(\Gu),
\] 
where $U(p^k)$ is the unitary group.


It is not hard to check that the standard representation $s_{p^k} \colon \Sigma_{p^k} \to U(p^k)$ fits into a commutative diagram
\[
\xymatrix{\Sect_{\Div_{p^k}(\Gu)} \ar[r]^-{\cong} \ar@{->>}[d]   &  E^0(BU(p^k)) \ar[d]^{s_{p^k}^*}\\  \Sect_{\Sub_{p^k}(\Gu)} \ar@{^{(}->}[d] & E^0(B\Sigma_{p^k})/I_{\tr}  \ar@{^{(}->}[d] \\ \Prod{\Sub_{p^k}(\qz)} C_0 \ar[r]^-{\cong}  &  \Cl_n(\Sigma_{p^k},C_0)/I_{\tr}.}
\]
The map 
\[
\Sect_{\Sub_{p^k}(\Gu)} \to \Prod{\Sub_{p^k}(\qz)} C_0 
\]
is just the map to the base change 
\[
\Sect_{\Sub_{p^k}(\Gu)} \to C_0 \otimes_{E^0} \Sect_{\Sub_{p^k}(\Gu)} \cong \Sect_{\Sub_{p^k}(\qz)} \cong \Prod{\Sub_{p^k}(\qz)} C_0.
\]
and it is an injection since $\Sect_{\Sub_{p^k}(\Gu)}$ is a finitely generated free $E^0$-module (Corollary \ref{ltsub}).

The map $E^0(B\Sigma_{p^k})/I_{\tr} \to \Cl_n(\Sigma_{p^k},C_0)/I_{\tr}$ is also induced by base change and it is an injection because Theorem 8.6 of \cite{etheorysym} states that $E^0(B\Sigma_{p^k})/I_{\tr}$ is a free $E^0$-module.

A diagram chase in the commutative diagram above gives us an injective map
\[
\Sect_{\Sub_{p^k}(\Gu)} \hookrightarrow E^0(B\Sigma_{p^k})/I_{\tr}.
\]
Proving that this map is an isomorphism requires some work. We refer the reader to \cite{subgroups} or \cite{genstrickland} to see two different ways that this can be accomplished.
\end{proof}


The isomorphism in Theorem \ref{strickland} is built in such a way that the following diagram commutes
\begin{equation} \label{subgroups}
\xymatrix{E^0(B\Sigma_{p^k})/I_{\tr} \ar[r]^-{\cong} \ar[d]  & \Sect_{\Sub_{p^k}(\Gu)} \ar[d] \\ \Cl_n(\Sigma_{p^k},C_0)/I_{\tr} \ar[r]^-{\cong} \ar[d]_-{\pi_H} & \Prod{\Sub_{p^k}(\qz)} C_0 \ar[d]^-{\pi_H} \\ C_0 \ar[r]^{=} & C_0.}
\end{equation}
The composite of the left vertical arrows is $\chi_H$ from Equation \ref{chih}. The composite of the right vertical arrows factors through $\Sect_{\Level(\qz,\Gu)}$ and the resulting map $\Sect_{\Sub_{p^k}(\Gu)} \to \Sect_{\Level(\qz,\Gu)}$ depending on $H \subset \qz$ is the map of moduli problems sends
\[
(l,i,\tau) \mapsto (l(H) \subset \G, i, \tau).
\]

Recall that the additive power operation applied to $BA$ for $A$ a finite abelian group is the ring map
\[
P_{p^k}/I_{\tr} \colon E^0(BA) \to E^0(BA) \otimes_{E^0} E^0(B\Sigma_{p^k})/I_{\tr}.
\]
The domain and codomain both admit interpretations in terms of moduli problems over Lubin-Tate space. Applying $\Spf(-)$ gives a map
\[
\Sub_{p^k}(\Gu) \times_{\LT} \Hom(A^*,\Gu) \to \Hom(A^*,\Gu).
\]
There is an obvious guess for what this map might do when applied to a complete local ring $R$. An object in the groupoid
\[
\Sub_{p^k}(\Gu)(R) \times_{\LT(R)} \Hom(A^*,\Gu)(R)
\] 
is a deformation of $\F$ to $R$, $(\G, i, \tau)$, equipped with a subgroup scheme of order $p^k$, $H \subset \G$ and a homomorphism $A^* \to \G$. We need to produce a deformation of $\F$ equipped with a map from $A^*$. This can be accomplished by taking the composite
\[
A^* \to \G \to \G/H
\]
and recalling that $\G/H$ is a deformation in a canonical way. This is the content of Proposition 3.21 of \cite{ahs}:
\begin{thm} \cite[Proposition 3.21]{ahs} \label{ahs}
The power operation $P_{p^k}/I_{\tr}$ is the ring of functions on the map of moduli problems
\[
\Sub_{p^k}(\Gu) \times_{\LT} \Hom(A^*,\Gu) \to \Hom(A^*,\Gu)
\]
that, when applied to a complete local ring $(R,m)$, sends a deformation equipped with subgroup of order $p^k$ and homomorphism $A^* \to \G$
\[
(H \subset \G, A^* \to \G, i, \tau)
\]
to the tuple
\[
(A^* \to \G \to \G/H, i \circ \sigma^k, \tau/H).
\]
\end{thm}

This theorem can be used to understand the relationship between $P_{p^k}/I_{\tr}$ applied to finite abelian groups and the character maps. Recall the construction of $\chi_{[\al]} \colon E^0(BG) \to C_0$ in Equation \eqref{charactermap}.

\begin{prop} \label{AHSrestricted}
Let $H \subset \qz$ be a subgroup of order $p^k$, let $\al \colon \lat \to A$ be a group homomorphism, and let $\phi_H \colon \qz \to \qz$ be an isogeny such that $\ker(\phi_H) = H$. There is a commutative diagram
\[
\xymatrix{E^0(BA) \ar[r]^-{P_{p^k}/I_{\tr}} \ar[d]_{\chi_{[\al \phi_{H}^{*}]}} & E^0(BA) \otimes_{E^0} E^0(B\Sigma_{p^k})/I_{\tr} \ar[d]^{\chi_{[\al]} \otimes \chi_H} \\ C_0 \ar[r]^-{\phi_H} & C_0.}
\]
\end{prop}
\begin{proof}
Since the image of the character map lands in $\Sect_{\Level(\qz,\Gu)}$, it suffices to prove that
\[
\xymatrix{E^0(BA) \ar[r]^-{P_{p^k}/I_{\tr}} \ar[d]_{\chi_{[\al \phi_{H}^{*}]}} & E^0(BA) \otimes_{E^0} E^0(B\Sigma_{p^k})/I_{\tr} \ar[d]^{\chi_{[\al]} \otimes \chi_H} \\ \Sect_{\Level(\qz,\Gu)} \ar[r]^-{\phi_H} & \Sect_{\Level(\qz,\Gu)}.}
\]
commutes. Recall the definition of $\psi_H$ from Diagram \eqref{psidiagram}. Now applying Example \ref{characterab} to the maps $\chi_{[\al]}$ and $\chi_{[\al \phi_{H}^*]}$, applying the discussion around Diagram \eqref{subgroups} to $\chi_H$, and applying Theorem \ref{ahs} to the top arrow, we just need to check that a certain diagram of moduli problems commutes. Going around the bottom direction gives
\[
(l, i, \tau) \mapsto ((l/H)\psi_{H}^{-1},i \sigma^{k},\tau/H) \mapsto ((l/H)\psi_{H}^{-1}(\al \phi_{H}^*)^*,i \sigma^{k},\tau/H)
\]
and going around the top direction gives
\[
(l, i, \tau) \mapsto (l \al^*, i, \tau) \times (H \subset \G, i, \tau) \mapsto (A^* \lra{l \al^*} \G \to \G/H, i \sigma^{k},\tau/H).
\]
We want to show that the two resulting deformations are equal. Taking the quotient with respect to $H \subset \qz$ and its image in $\G$ the level structure $l$ gives the commutative diagram
\[
\xymatrix{\qz \ar[r]^l \ar[d]_{q_H} & \G \ar[d] \\ \qz/H \ar[r]^{l/H} & \G/H.}
\]
Thus the composite $A^* \lra{l \al^*} \G \to \G/H$ is equal to $(l/H)q_H\al^*$. Now
\[
(l/H)\psi_{H}^{-1}(\al \phi_{H}^*)^* = (l/H)\psi_{H}^{-1}\phi_{H}^*\al^* = (l/H)q_H\al^*
\]
by Diagram \eqref{psidiagram}, so we are finished.
\end{proof}


\section{The character of the power operation} \label{secmain} 

With all of these tools in hand, we can finally get down to business. We'd like to construct a ``power operation" on generalized class functions that is compatible with the total power operation on $E$ through the character map of Hopkins, Kuhn, and Ravenel. We can say this diagramatically: we'd like to construct an operation
\[ 
\Cl_n(G,C_0) \to \Cl_n(G \times \Sigma_m, C_0)
\]
making the diagram
\[
\xymatrix{E^0(BG) \ar[r]^-{P_m} \ar[d]_-{\chi} & E^0(BG \times B\Sigma_m) \ar[d]^-{\chi} \\ \Cl_n(G,C_0) \ar[r]^{} & \Cl_n(G \times \Sigma_m, C_0)}
\]
commute. Why would we like to do this? Since the ring of generalized class functions is just the $C_0$-valued functions on a set, such a formula should have a particularly simple form. Ideally it would only include information from group theory and information about endomorphisms of the ring $C_0$ and this turns out to be the case. Further, Theorem \ref{thmc} above implies that class functions knows quite a bit about $E^0(BG)$. For instance, the rationalization of $E^0(BG)$ sits inside class functions as the $\Aut(\qz)$-invariants. We will begin by describing the formula for the power operation on class functions as well as some of its consequences and then we will give a description of the proof of the fact that it is compatible with the power operations for $E$-theory.

We will produce more than one power operation on generalized class functions. Let $\Sub(\qz)$ be the set of finite subgroups of $\qz$. Let $\pi \colon \Isog(\qz) \to \Sub(\qz)$ be the surjective map sending an endoisogeny of $\qz$ to its kernel. The map $\pi$ makes $\Isog(\qz)$ into an $\Aut(\qz)$-torsor over $\Sub(\qz)$ and we will write $\Gamma(\Sub(\qz), \Isog(\qz))$ for the set of sections. For each $\phi \in \Gamma(\Sub(\qz), \Isog(\qz))$,
\[
\xymatrix{\Isog(\qz) \ar[d]_{\pi} \\ \Sub(\qz), \ar@/_/[u]_{\phi}} 
\]
we will produce an operation
\[
P_{m}^{\phi} \colon \Cl_n(G, C_0) \to \Cl_n(G \times \Sigma_m, C_0)
\] 
compatible with the power operation on $E$ through the character map. 


Given a subgroup $H \subset \qz$, let $\phi_H \colon \qz \to \qz$ be the corresponding isogeny of $\qz$, so $\phi_H \in \Isog(\qz)$ with $\ker(\phi_H) = H$. Let $f \in \Cl_n(G, C_0)$ be a generalized class function. To define $P_{m}^{\phi}$ we only need to give a value for $P_{m}^{\phi}(f)$ on a conjugacy class $[\lat \to G \times \Sigma_m]$. Proposition \ref{conjiso} implies that this conjugacy class corresponds to a pair $([\al \colon \lat \to G],\oplus H_i)$. We set
\[
P_{m}^{\phi}(f)([\al],\Oplus{i} H_i) = \Prod{i}\big (f([\al \circ \phi_{H_i}^{*}])\phi_{H_i}\big ).
\]
This formula may look like it has come out of the blue. Let us verify that it makes sense in the case that we understand best:

\begin{prop} \label{AHSchar}
Let $A$ be a finite abelian group and let $\phi \in \Gamma(\Sub(\qz), \Isog(\qz))$, then there is a commutative diagram
\[
\xymatrix{E^0(BA) \ar[rr]^-{P_{p^k}/I_{\tr}} \ar[d]_{\chi} & & E^0(BA) \otimes_{E^0} E^0(B\Sigma_{p^k})/I_{\tr} \ar[d]^{\chi \otimes \chi} \\ \Cl_n(A,C_0) \ar[rr]^-{P_{p^k}^{\phi}/I_{\tr}} & & \Cl_n(A,C_0) \otimes_{C_0} \Cl_n(\Sigma_{p^k},C_0)/I_{\tr}.}
\]
\end{prop}
\begin{proof}
Let $([\al \colon \lat \to A],H)$ correspond to a conjugacy class $[\lat \to A \times \Sigma_{p^k}]$ which is transitive on the factor $\Sigma_{p^k}$. Let $f \in \Cl_n(A,C_0)$, then the definition of $P_{p^k}^{\phi}/I_{\tr}$ applied to $f$ is
\[
P_{p^k}^{\phi}/I_{\tr}(f)([\al],H) = f([\al \phi_{H}^*])\phi_H.
\]
Thus the value of $P_{p^k}^{\phi}/I_{\tr}(f)$ on the conjugacy class $([\al],H)$ is determined by the value of $f$ on $[\al \phi_{H}^*]$. Hence, it suffices to prove that the diagram
\[
\xymatrix{E^0(BA) \ar[r]^-{P_{p^k}/I_{\tr}} \ar[d]_{\chi_{[\al \phi_{H}^{*}]}} & E^0(BA) \otimes_{E^0} E^0(B\Sigma_{p^k})/I_{\tr} \ar[d]^{\chi_{[\al]} \otimes \chi_H} \\ C_0 \ar[r]^-{\phi_H} & C_0}
\]
commutes. This is Proposition \ref{AHSrestricted}.
\end{proof}


Now that we have checked our sanity, we will prove that these operations satisfy basic properties that we've come to expect from operations called power operations. We give a complete proof of the next lemma in order to help the reader get used to manipulating $P_{m}^{\phi}$.

\begin{lemma}
For all $\phi\in \Gamma(\Sub(\qz), \Isog(\qz))$, the power operation $P_{m}^{\phi}$ is natural in the group variable.
\end{lemma}
\begin{proof}
Let $\gamma \colon G \to K$ be a group homomorphism so that 
\[
\gamma^* \colon \Cl_n(K,C_0) \to \Cl_n(G,C_0)
\]
is defined by $(\gamma^* f)([\al]) = f([\gamma \al])$ for $[\al \colon \lat \to G]$. We wish to show that 
\[
P_{m}^{\phi}(\gamma^* f) = (\gamma \times \id_{\Sigma_m})^*P_{m}^{\phi}(f).
\]
This follows because
\begin{align*}
P_{m}^{\phi}(\gamma^* f)([\al], \oplus_i H_i) &= \Prod{i} (\gamma^* f)([\al \phi_{H_i}^*]) \phi_{H_i} \\
&= \Prod{i} f([\gamma \al \phi_{H_i}^*]) \phi_{H_i} \\
&= P_{m}^{\phi}(f)([\gamma \al], \oplus_i H_i) \\
&= (\gamma \times \id_{\Sigma_m})^*P_{m}^{\phi}(f).
\end{align*}
\end{proof}

Given two power operations on class function $P^{\phi}_{i}$ and $P^{\phi}_{j}$, their external product is the map
\[
P^{\phi}_{i} \times P^{\phi}_{j} \colon \Cl_n(G,C_0) \to \Cl_n(G \times \Sigma_i \times G \times \Sigma_j, C_0)
\]
given by the formula
\[
(P^{\phi}_{i} \times P^{\phi}_{j})(f)(([\al_1], \Oplus{i}H_i),([\al_2],\Oplus{i}K_i)) = P^{\phi}_{i}(f)([\al_1], \Oplus{i}H_i)P^{\phi}_{j}(f)([\al_2], \Oplus{i}K_i).
\]
Let $\Delta \colon G \to G \times G$ be the diagonal map. Restricting $P^{\phi}_{i} \times P^{\phi}_{j}$ along the diagonal gives us a map
\[
\Delta^*(P^{\phi}_{i} \times P^{\phi}_{j}) \colon \Cl_n(G,C_0) \to \Cl_n(G \times G \times \Sigma_i \times \Sigma_j, C_0)
\]
that we will denote by $\Delta^*(P^{\phi}_{i} \times P^{\phi}_{j})$. The following lemma is an analogue of the identity involving power operations in Diagram \eqref{powerdiagram}. We leave the proof to the reader.

\begin{lemma} \label{powerdiagram2}
Let $\phi \in \Gamma(\Sub(\qz), \Isog(\qz))$ and let $i,j>0$ with $i+j = m$, then restriction along the inclusion $\Delta_{i,j} \colon \Sigma_i \times \Sigma_j \subset \Sigma_m$ induces the commutative diagram
\[
\xymatrix{\Cl_n(G,C_0) \ar[r]^-{P^{\phi}_{m}} \ar[dr]_<<<<<<{\Delta^*(P^{\phi}_{i} \times P^{\phi}_{j})} & \Cl_n(G \times \Sigma_m, C_0) \ar[d]^{\Delta_{i,j}^{*}} \\ & \Cl_n(G \times \Sigma_i \times \Sigma_j, C_0).}
\]
\end{lemma}

Finally we can state the result that we have been working towards:

\begin{thm} \label{mainthm}
Let $\phi \in \Gamma(\Sub(\qz), \Isog(\qz))$ and let $G$ be a finite group. There is a commutative diagram 
\[
\xymatrix{E^0(BG) \ar[r]^-{P_m} \ar[d]_{\chi} & E^0(BG \times B\Sigma_m) \ar[d]^{\chi} \\ \Cl_n(G,C_0) \ar[r]^-{P_{m}^{\phi}} & \Cl_n(G \times \Sigma_m, C_0).}
\]
\end{thm} 

It may seem unsatisfying that any choice of section gives us a power operation on generalized class functions that is compatible with the power operation on Morava $E$-theory. This is essentially a consequence of the fact that $C_0$ is an $\Aut(\qz)$-Galois extension of $\Q \otimes E^0$. It turns out that the choice disappears after taking the $\Aut(\qz)$-invariants of the ring of generalized class functions!

\begin{example}
As we have seen, when $(\F, \ka) = (\hat{\G}_m, \F_p)$, $E = K_p$. In this case $n=1$ and $\qz = \Q_p/\Z_p$. Hence there is a canonical element in $\Gamma(\Sub(\qz), \Isog(\qz))$, the section sending $\qz[p^k]$ to the multiplication by $p^k$ map on $\qz$. It turns out that, in this case, the action of these isogenies on $C_0$ is trivial. Thus the formula for $P_{m}^{\phi}$ simplifies for this choice of section. 
\end{example}

The proof of Theorem \ref{mainthm} has three steps. We will first reduce to the case that $m = p^k$, then we will reduce to the case that $G$ is abelian and finally we will perform an induction on $k$. Each step will require a certain map between class functions to be injective. We present these three injections now as three lemmas.

\begin{lemma} \label{lem1}
Let $\sum_j a_j p^j$ be the $p$-adic expansion of $m$. The inclusion
$\prod_{j} \Sigma_{p^j}^{\times a_j} \subseteq \Sigma_m$ induces an injection
\[
\Cl_n(\Sigma_{m},C_0) \hookrightarrow \bigotimes_{j} \Cl_n(\Sigma_{p^j},C_0)^{\otimes a_j}.
\]
\end{lemma}
\begin{proof}
This follows immediately from Corollary \ref{padicsum}.
\end{proof}

We leave the proof of the next lemma to the reader.
\begin{lemma} \label{lem2}
Let $G$ be a finite group. The product of the restriction maps to abelian subgroups of $G$ is an injection
\[
\Cl_n(G,C_0) \hookrightarrow \Prod{A \subset G} \Cl_n(A,C_0),
\]
where the product ranges over all abelian subgroups of $G$.
\end{lemma}

\begin{lemma} \label{lem3}
For all $k>0$, there is an injection
\[
\Cl_n(\Sigma_{p^{k}},C_0) \hookrightarrow \Cl_n(\Sigma_{p^{k-1}}^{\times p},C_0) \times \Cl_n(\Sigma_{p^{k}},C_0)/I_{\tr},
\]
where the map to the left factor is the restriction along $\Sigma_{p^{k-1}}^{\times p} \subset \Sigma_{p^k}$ and the map to the right factor is the quotient map.
\end{lemma}
\begin{proof}
Applying Proposition \ref{conjiso} and Corollary \ref{subiso} as well as Lemma \ref{summap}, we see that this map is the $C_0$-valued functions on the map of sets
\[
\Sub_{p^k}(\qz) \coprod \Sum_{p^{k-1}}(\qz)^{\times p} \to \Sum_{p^k}(\qz)
\]
sending a subgroup in the first component to itself and the $p$-tuple of sums of subgroups in the second component to their sum. This is a surjective map of sets.
\end{proof}

The rest of this section constitutes a proof of Theorem \ref{mainthm}. We begin with the reduction to $m = p^k$. Consider the commutative diagram 
\[
\xymatrix{E^0(BG) \ar[r]^-{P_m} \ar[d] & E^0(BG \times B\Sigma_m) \ar[r] \ar[d] & E^0(BG \times \prod_{j} B\Sigma_{p^j}^{a_j}) \ar[d] \\ \Cl_n(G,C_0)  \ar[r]^-{P_{m}^{\phi}}  & \Cl_n(G \times \Sigma_m, C_0) \ar@{^{(}->}[r] & \Cl_n(G \times \prod_{j} \Sigma_{p^j}^{a_j}, C_0).}
\]
The composite along the top is the external product of power operations $\Delta^*(\prod_{j}P_{p_j}^{\times a_j})$. The bottom composite is the external product $\Delta^*(\prod_{j}(P_{p_j}^{\phi})^{\times a_j})$. Since the bottom right arrow is an injection by Lemma \ref{lem1}, to prove that the left hand square commutes, it suffices to prove the large square commutes. The large square is built only using power operations of the form $P_{p^k}$ and $P_{p^k}^{\phi}$.

Now we describe the reduction to finite abelian groups. Assume that the diagram
\[
\xymatrix{E^0(BA) \ar[r]^-{P_{p^k}} \ar[d]_{\chi} & E^0(BA \times B\Sigma_{p^k}) \ar[d]^{\chi} \\ \Cl_n(A,C_0) \ar[r]^-{P_{p^k}^{\phi}} & \Cl_n(A \times \Sigma_{p^k}, C_0)}
\]
commutes for any choice of $\phi$ and any finite abelian group. Consider the cube
\[\resizebox{\columnwidth}{!}{
\xymatrix{& E^0(BG) \ar[ld]_{P_{p^k}} \ar'[d][dd] \ar[rr] & & \Prod{A \subseteq G} E^0(BA) \ar[ld]^{\prod P_{p^k}} \ar[dd] \\
E^0(BG \times B\Sigma_{p^k}) \ar[dd] \ar[rr] & &  \Prod{A\subseteq G} E^0(BA \times B\Sigma_{p^k}) \ar[dd] \\
& \Cl_n(G,C_0) \ar[ld]_{P_{p^k}^{\phi}} \ar@{^{(}->}'[r][rr] & &  \Prod{A \subseteq G}\Cl_n(A,C_0) \ar[ld]^{\prod P_{p^k}^{\phi}} \\
\Cl_n(G \times \Sigma_{p^k},C_0) \ar@{^{(}->}[rr]  &&  \Prod{A\subseteq G}\Cl_n(A \times \Sigma_{p^k},C_0).}
}\]
We do not claim that this cube commutes yet. Since Lemma \ref{lem2} implies that the horizontal arrows of the bottom face are injections, a diagram chase shows that the left face commutes if the right face commutes. This reduces us to the case of finite abelian groups.

Now we come to the inductive part of the proof. The base case of our induction, $k=0$, is the following proposition:

\begin{prop}
For each section $\phi \in \Gamma(\Sub(\qz), \Isog(\qz))$, there is a commutative diagram
\[
\xymatrix{E^0(BA) \ar[r]^-{P_1} \ar[d]_{\chi} & E^0(BA) \ar[d]^{\chi} \\ \Cl_n(A,C_0) \ar[r]^-{P_{1}^{\phi}} & \Cl_n(A, C_0).}
\]
\end{prop}
\begin{proof}
Recall that $P_1$ is the identity map and let $e \subset \qz$ be the trivial subgroup so $\phi_e$ is an automorphism of $\qz$. The power operation $P_{1}^{\phi}$ is defined by
\[
P_{1}^{\phi}(f)([\al]) = f([\al \phi_e])\phi_e.
\]
Thus, working one factor at a time, it suffices to show that the diagram
\[
\xymatrix{E^0(BA) \ar[r]^-{\id} \ar[d]_{\chi_{[\al \phi_{e}^*]}} & E^0(BA) \ar[d]^{\chi_{[\al]}} \\ \Sect_{\Level(\qz,\Gu)} \ar[r]^{\phi_e} &\Sect_{\Level(\qz,\Gu)}}
\]
commutes. Now the proof follows the same lines as the proof of Proposition \ref{AHSrestricted}. Since $H=e$, the map $\psi_H$ of Diagram \eqref{psidiagram} is just $\phi_e$. Applying Example \ref{characterab} to each of the vertical arrows and Equation \eqref{isogaction} to the bottom arrow, we see that, on the level of moduli problems, going around the bottom direction gives
\[
(l, i, \tau) \mapsto (l\phi_{e}^{-1},i,\tau) \mapsto (l\phi_{e}^{-1}(\al \phi_{e}^*)^*,i,\tau)
\]
and going around the top direction gives
\[
(l, i, \tau) \mapsto (l \al^*, i, \tau).
\]
Since $l\phi_{e}^{-1}(\al \phi_{e}^*)^* = l\phi_{e}^{-1} \phi_{e} \al^* = l\al^*$, we are done.
\end{proof}

We use the following commutative diagram in order to be able to apply our induction hypothesis 
\[
\xymatrix{E^0(BA \times B\Sigma_{p^k}) \ar[d]^{\chi} \ar[r] & E^0(BA \times B\Sigma_{p^{k-1}}^{\times p}) \times (E^0(BA) \otimes_{E^0} E^0(B\Sigma_{p^{k}})/I_{\tr})  \ar[d] \\  \Cl_n(A \times \Sigma_{p^k}, C_0) \ar@{^{(}->}[r] & \Cl_n(A \times \Sigma_{p^{k-1}}^{\times p}, C_0) \times (\Cl_n(A,C_0) \otimes_{C_0} \Cl_n(\Sigma_{p^k}, C_0)/I_{\tr}),}
\]
where both of the horizontal arrows are the product of the restriction and quotient maps and the bottom horizontal arrow is an injection by Lemma \ref{lem3}.

Now Diagram \eqref{powerdiagram} and Lemma \ref{powerdiagram2} imply that it is suffices to prove that the two diagrams
\[
\xymatrix{E^0(BA) \ar[rr]^-{\Delta^*(P_{p^{k-1}}^{\times p})} \ar[d] & & E^0(BA) \otimes_{E^0} E^0(B\Sigma_{p^{k-1}}^{\times p}) \ar[d] \\ \Cl_n(A,C_0) \ar[rr]^-{\Delta^*((P_{p^{k-1}}^{\phi})^{\times p})} & & \Cl_n(A,C_0) \otimes_{C_0} \Cl_n(\Sigma_{p^{k-1}}^{\times p},C_0)}
\]
and
\[
\xymatrix{E^0(BA) \ar[r]^-{P_{p^k}/I_{\tr}} \ar[d] & E^0(BA) \otimes_{E^0} E^0(B\Sigma_{p^k})/I_{\tr} \ar[d] \\ \Cl_n(A,C_0) \ar[r]^-{P_{p^k}^{\phi}/I_{\tr}} & \Cl_n(A,C_0) \otimes_{C_0} \Cl_n(\Sigma_{p^k},C_0)/I_{\tr}}
\]
commute. The first commutes by the induction hypothesis and the second commutes by Proposition \ref{AHSchar}.

\section{Loose ends}
Although the author prefers burnt ends to loose ends, in this final section we will try to wrap up some loose ends. We will describe the restriction of $P_{m}^{\phi}$ to $\Aut(\qz)$-invariant class functions, we will discuss the relationship between the power operations $P_{m}^{\phi}$ and the stabilizer group action, and we will describe some of the changes necessary to produce a total power operation on class functions.

In order to show that $\Aut(\qz)$-invariant class functions are sent to $\Aut(\qz)$-invariant class functions by $P_{m}^{\phi}$, we need the following lemma that describes the $\Aut(\qz)$-action induced on $\Sum_m(\qz)$ through the isomorphism of Proposition \ref{conjiso}.

In order to show that $\Aut(\qz)$-invariant class functions are sent to $\Aut(\qz)$-invariant class functions by $P_{m}^{\phi}$, we need to understand the $\Aut(\qz)$-action induced on $\Sum_m(\qz)$ through the isomorphism of Proposition \ref{conjiso}. The automorphism $\gamma \in \Aut(\qz)$ acts on $\hom(\lat,\Sigma_m)_{/ \sim}$ by $\gamma \cdot [\al] = [\gamma^* \al]$. The corresponding action of $\Aut(\qz)$ on $\Sum_{m}(\qz)$ through the isomorphism of Proposition \ref{conjiso} is given by
\[
\gamma \cdot \oplus_i H_i = \oplus_i \gamma H_i.
\]
To see this, recall from the proof of Proposition \ref{conjiso} that $H_i = (\lat/\lat_i)^* \subset \qz$, where $\lat_i$ is the stabilizer of the transitive $\lat$-set $\um_i$. The stabilizer of the $\lat$-set structure on $\um_i$ given by acting through $\gamma^*$ is $(\gamma^*)^{-1}\lat_i$ and there is a canonical isomorphism
\[
(\lat/(\gamma^*)^{-1}\lat_i)^* \cong \gamma((\lat/\lat_i)^*) = \gamma H_i.
\]

\begin{prop} \label{invariants}
For any section $\phi \in \Gamma(\Sub(\qz), \Isog(\qz))$, the power operation 
\[
P_{m}^{\phi} \colon \Cl_n(G,C_0) \to \Cl_n(G \times \Sigma_m,C_0)
\]
sends $\Aut(\qz)$-invariant class functions to $\Aut(\qz)$-invariant class functions and the resulting map
\[
\Cl_n(G,C_0)^{\Aut(\qz)} \to \Cl_n(G \times \Sigma_m,C_0)^{\Aut(\qz)}
\]
is independent of the choice of $\phi$.
\end{prop}
\begin{proof}
This is a calculation. Recall that if $f \in \Cl_n(G, C_0)$ and $\gamma \in \Aut(\qz)$, then
\[
(f \cdot \gamma)([\al]) =  (f([\al \gamma^*]))\gamma.
\]
Let $f \in \Cl_n(G,C_0)^{\Aut(\qz)} $, let $\phi$ be a section, and let $\gamma \in \Aut(\qz)$, then
\begin{align*}
(P_{m}^{\phi}(f) \cdot \gamma)([\al],\oplus H_i) &= P_{m}^{\phi}(f)(\gamma \cdot ([\al],\oplus H_i))\gamma \\
&=  P_{m}^{\phi}(f)( ([\al  \gamma^*],\oplus \gamma H_i))\gamma\\
&=  \big (\Prod{i} f([\al  \gamma^* \phi_{\gamma H_i}^*])\phi_{\gamma H_i} \big ) \gamma\\
&= \Prod{i} \big ( f([\al  \gamma^* \phi_{\gamma H_i}^*])\phi_{\gamma H_i}  \gamma \big ).
\end{align*}
Now notice that, for each subgroup $H_i$, there is an automorphism $\sigma_i$ making the following diagram
\[
\xymatrix{\qz \ar[r]^{\gamma} \ar[d]_{\phi_{H_i}} & \qz \ar[d]^{\phi_{\gamma H_i}} \\ \qz \ar[r]^{\sigma_i} & \qz}
\]
commute. This implies that $\phi_{\gamma H_i} \gamma = \sigma_i \phi_{H_i}$. Applying this to the last equality gives
\begin{align*}
\Prod{i} \big ( f([\al  \gamma^* \phi_{\gamma H_i}^*])\phi_{\gamma H_i}  \gamma \big ) &= \Prod{i} \big ( f([\al  \phi_{H_i}^* \sigma_{i}^*]) \sigma_i \phi_{H_i} \big ) \\
&= \Prod{i}\big ( f([\al  \phi_{H_i}^*])  \phi_{H_i} \big ) \\
&= P_{m}^{\phi}(f)([\al], \oplus H_i),
\end{align*}
where the second equality follows from the fact that $f$ is $\Aut(\qz)$-invariant.

Now assume that $\phi, \xi \in \Gamma(\Sub(\qz), \Isog(\qz))$ are two sections. Then for each $H \subset \qz$ there exists $\sigma_H \in \Aut(\qz)$ such that $\xi_H = \sigma_H \phi_H$. Now assume $f \in \Cl_n(G,C_0)^{\Aut(\qz)}$, then
\begin{align*}
P_{m}^{\xi}(f)([\al],\oplus H_i) &= \Prod{i} f([\al \xi_{H_i}^*]) \xi_{H_i} \\
&= \Prod{i} f([\al \phi_{H_i}^* \sigma_{H_i}^*]) \sigma_{H_i} \phi_{H_i} \\
&= \Prod{i} f([\al \phi_{H_i}^*])\phi_{H_i}\\
&= P_{m}^{\phi}(f)([\al],\oplus H_i).
\end{align*}
This proves that the resulting power operation
\[
\Cl_n(G,C_0)^{\Aut(\qz)} \to \Cl_n(G \times \Sigma_m,C_0)^{\Aut(\qz)}
\]
is independent of the choice of section.
\end{proof}

Theorem \ref{thmc} states that there is an isomorphism
\[
\Q \otimes E^0(BG) \cong \Cl_n(G,C_0)^{\Aut(\qz)}.
\]
Thus we have produced a ``rational power operation"
\[
P_{m}^{\Q} \colon \Q \otimes E^0(BG) \to \Q \otimes E^0(BG \times B\Sigma_m)
\]
and, at the same time, given a formula for it.

\begin{prop}
For any $\phi  \in \Gamma(\Sub(\qz), \Isog(\qz))$, the diagonal action of the stabilizer group $\Aut(\F/\ka)$ on generalized class functions commutes with the power operation $P_{m}^{\phi}$. That is, for $s \in \Aut(\F/\ka)$, there is a commutative diagram
\[
\xymatrix{\Cl_n(G,C_0) \ar[r]^-{P_{m}^{\phi}} \ar[d]_{s} & \Cl_n(G \times \Sigma_m,C_0) \ar[d]^{s} \\ \Cl_n(G,C_0) \ar[r]^-{P_{m}^{\phi}} & \Cl_n(G \times \Sigma_m,C_0).}
\]
\end{prop}
\begin{proof}
Recall that $\Aut(\F/\ka)$ acts on $C_0$ on the right through ring maps and that $\Isog(\qz)$ acts on $C_0$ on the right through ring maps. Tracing through the diagram in the proposition, we want to show that, for $f \in \Cl_n(G,C_0)$ and $s \in \Aut(\F/\ka)$,
\[
\Prod{i} f([\al \phi_{H_i}^*])s \phi_{H_i} = \Prod{i} f([\al \phi_{H_i}^*]) \phi_{H_i}s.
\]
Thus it suffices to show that the actions of $\Aut(\F/\ka)$ and $\Isog(\qz)$ on $C_0$ commute. Since $C_0 = \Q \otimes \Sect_{\Level(\qz,\Gu)}$, it suffices to show that the actions commute on $\Sect_{\Level(\qz,\Gu)}$ and this is a question about moduli problems over Lubin-Tate space. 

Recall the formulas of Equations \eqref{stabaction} and \eqref{isogaction}. Since $\Aut(\F,\ka)$ does not affect the data of $\G$ or $i$ in a deformation, it suffices to show that
\[
((i \sigma^k)^*s)(\tau/H)  = (i^*s \tau)/H
\]
and this is a statement about isogenies of formal groups over $\ka$. Unwrapping this equality, we want to show that the diagram
\[
\xymatrix{ \pi^* \G \ar[r]^-{\tau} \ar[d] & i^*\F \ar[r]^-{i^*s} \ar[d]_{i^*\Frob^k} & i^*\F  \ar[d]^{i^*\Frob^k} \\ (\pi^* \G)/H \ar[r]_-{\tau/H} & (i \sigma^k)^* \F \ar[r]_-{(i \sigma^k)^*s} & (i \sigma^k)^* \F  }
\]
commutes. The left hand square commutes by construction and the right hand square is $i^*$ applied to the square
\[
\xymatrix{\F \ar[r]^-{s} \ar[d]_-{\Frob^k} & \F \ar[d]^-{\Frob^k} \\ (\sigma^k)^*\F \ar[r]^-{(\sigma^k)^*s} & (\sigma^k)^*\F.}
\]
It is not hard to see that this commutes. Choose a coordinate so that we have formal group laws $x +_{\F} y$ and $x+_{(\sigma^k)^*\F} y$ and automorphisms of formal group laws $f_s(x)$ and $f_{(\sigma^k)^*s}(x)$ such that
\[
(x +_{\F} y)^{p^k} = x^{p^k}+_{(\sigma^k)^*\F} y^{p^k}
\]
and
\[
(f_s(x))^{p^k} = f_{(\sigma^k)^*s}(x^{p^k}).
\]
Then the induced diagram of $\ka$-algebras
\[
\xymatrix{\ka \powser{x} & \ka \powser{x} \ar[l] \\ \ka \powser{x} \ar[u] & \ka \powser{x} \ar[l] \ar[u]}
\]
sends
\[
\xymatrix{f_{(\sigma^k)^*s}(x^{p^k}) = (f_{s}(x))^{p^k} & x^{p^k} \ar@{|->}[l] \\ f_{(\sigma^k)^*s}(x) \ar@{|->}[u] & x \ar@{|->}[l] \ar@{|->}[u]}
\]
and thus commutes.
%
\end{proof}

Finally, we describe a few of the changes necessary to construct a total power operation, rather than just a power operation, on generalized class functions and prove that it is compatible with the total power operation on Morava $E$-theory. Just as in the case of $P_{m}^{\phi}$, the construction of the total power operation on generalized class functions
\[
\Po_{m}^{\phi} \colon \Cl_n(G,C_0) \to \Cl_n(G \wr \Sigma_m,C_0)
\]
depends on the choice of a section $\phi \in \Gamma(\Sub(\qz), \Isog(\qz))$. 

Recall Diagram \eqref{psidiagram}, it plays a more important role in the definition of $\Po_{m}^{\phi}$. For $H \subset \qz$, let $\lat_H = (\qz/H)^*$ so that we have an inclusion $q_{H}^* \colon \lat_H \hookrightarrow \lat$. 

Consider the set
\[
\Sum_m(\qz,G) = \{\oplus_i(H_i, [\al_i])| H_i \subset \qz, \sum_i|H_i|=m, \text{ and } [\al_i \colon \lat_{H_i} \to G]\}.
\]
When $G = e$, this is $\Sum_m(\qz)$. Proposition \ref{conjiso} generalizes to give the following proposition.
\begin{prop}
There is a canonical bijection
\[
\hom(\lat, G \wr \Sigma_m)_{/\sim} \cong \Sum_m(\qz,G).
\]
\end{prop}

Using this proposition, we define $\Po_{m}^{\phi}$ as follows: 
\[
\Po_{m}^{\phi}(f)(\oplus_i(H_i,[\al_i])) = \Prod{i} f([\al_i \psi_{H_i}^{*}])\phi_{H_i}
\]
for $f \in \Cl_n(G,C_0)$.

The main result of \cite{cottpo} is the following:
\begin{thm} \label{main2} \cite[Theorem 9.1]{cottpo}
Let $\phi \in \Gamma(\Sub(\qz), \Isog(\qz))$ and let $G$ be a finite group. There is a commutative diagram 
\[
\xymatrix{E^0(BG) \ar[r]^-{\Po_m} \ar[d]_{\chi} & E^0(BG \wr \Sigma_m) \ar[d]^{\chi} \\ \Cl_n(G,C_0) \ar[r]^-{\Po_{m}^{\phi}} & \Cl_n(G \wr \Sigma_m, C_0).}
\]
\end{thm} 

To prove Theorem \ref{main2}, we need generalizations of the ingredients that were required for the proof of Theorem \ref{mainthm}. We require a generalization of Theorem \ref{ahs} to the additive total power operation (Equation \eqref{additivetotal}) applied to abelian groups:
\[
\Po_m \colon E^0(BA) \to E^0(BA\wr\Sigma_m)/J_{\tr}.
\]
This is provided by \cite[Theorem 8.4]{cottpo}. It makes use of a generalization of Theorem \ref{strickland} to rings of the form $E^0(BA \wr \Sigma_{p^k})/J_{\tr}$, which is the main result of \cite{genstrickland}.

%
%
%
\bibliographystyle{amsalpha}
\bibliography{mybib}



\end{document}